\documentclass[12pt]{article}

\usepackage{geometry}
\usepackage{hyperref} 

\usepackage{amsmath}
\usepackage{amssymb}
\usepackage{amscd}
\usepackage{textcomp}
\usepackage{dsfont}
\usepackage{mathrsfs}

\long\def\forget#1{}

\newcommand{\Verkuerzung}[2]{#1}

\newcommand{\lang}[1]{\mbox{#1}}

\usepackage{color}
\newcounter{commentcounter}

\newcounter{urscommentcounter}

\def\?{\ 
{\bf\color{red}???}\ 
\immediate\write16{}
\immediate\write16{Warning: There was still a question mark . . . }
\immediate\write16{}}

\usepackage{amsthm}
\theoremstyle{plain}
\newtheorem{theorem}{Theorem}[section]

\newtheorem{corollary}[theorem]{Corollary}
\newtheorem{proposition}[theorem]{Proposition}

\theoremstyle{definition}
\newtheorem{definition}[theorem]{Definition}
\newtheorem{definition-theorem}[theorem]{Definition-Theorem}
\newtheorem{definition-remark}[theorem]{Definition-Remark}
\newtheorem{point}[theorem]{}

\newtheorem{remark}[theorem]{Remark}

\theoremstyle{remark}

\usepackage{xy}
\xyoption{all}

\newdir^{ (}{{}*!/-3pt/\dir^{(}}    
\newdir^{  }{{}*!/-3pt/\dir^{}}    
\newdir_{ (}{{}*!/-3pt/\dir_{(}}    
\newdir_{  }{{}*!/-3pt/\dir_{}}

\newcounter{zahl}

\def\theenumi{(\alph{enumi})}

\def\p@enumii{\theenumi}

\newcommand{\DS}{\displaystyle}
\newcommand{\TS}{\textstyle}
\newcommand{\SC}{\scriptstyle}
\newcommand{\SSC}{\scriptscriptstyle}

\DeclareMathOperator{\Aut}{Aut}

\DeclareMathOperator{\Gal}{Gal}

\DeclareMathOperator{\Koh}{H}

\DeclareMathOperator{\Hom}{Hom}

\DeclareMathOperator{\Isom}{Isom}

\DeclareMathOperator{\Quot}{Frac}

\DeclareMathOperator{\SL}{SL}

\DeclareMathOperator{\Spec}{Spec}
\DeclareMathOperator{\Spf}{Spf}

\DeclareMathOperator{\Tr}{Tr}

\newcommand{\alg}{{\rm alg}}

\newcommand{\fppf}{{\it fppf\/}}
\newcommand{\fpqc}{{\it fpqc\/}}

\DeclareMathOperator{\id}{\,id}

\newcommand{\red}{{\rm red}}

\newcommand{\sep}{{\rm sep}}

\newcommand{\scrA}{{\mathscr{A}}}

\newcommand{\scrH}{{\mathscr{H}}}

\newcommand{\scrS}{{\mathscr{S}}}

\DeclareMathOperator{\whtimes}{\mathchoice
            {\wh{\raisebox{0ex}[0ex]{$\DS\times$}}}
            {\wh{\raisebox{0ex}[0ex]{$\TS\times$}}}
            {\wh{\raisebox{0ex}[0ex]{$\SC\times$}}}
            {\wh{\raisebox{0ex}[0ex]{$\SSC\times$}}}}

\renewcommand{\phi}{\varphi}
\renewcommand{\epsilon}{\varepsilon}

\usepackage{amsfonts}
\newcommand{\BOne} {{\mathchoice{\hbox{\rm1\kern-2.7pt l\kern.9pt}}
                              {\hbox{\rm1\kern-2.7pt l\kern.9pt}}
                              {\hbox{\scriptsize\rm1\kern-2.3pt l\kern.4pt}}
                              {\hbox{\scriptsize\rm1\kern-2.4pt l\kern.5pt}}}}

\newcommand{\BA}{{\mathbb{A}}}

\newcommand{\BD}{{\mathbb{D}}}

\newcommand{\BF}{{\mathbb{F}}}
\newcommand{\BG}{{\mathbb{G}}}

\newcommand{\BL}{{\mathbb{L}}}

\newcommand{\BN}{{\mathbb{N}}}

\newcommand{\BP}{{\mathbb{P}}}
\newcommand{\BQ}{{\mathbb{Q}}}
\newcommand{\BR}{{\mathbb{R}}}
\newcommand{\BS}{{\mathbb{S}}}

\newcommand{\BU}{{\mathbb{U}}}

\newcommand{\BZ}{{\mathbb{Z}}}

\newcommand{\CF}{{\cal{F}}}
\newcommand{\CG}{{\cal{G}}}

\newcommand{\CI}{{\cal{I}}}

\newcommand{\CL}{{\cal{L}}}
\newcommand{\CM}{{\cal{M}}}
\newcommand{\CN}{{\cal{N}}}
\newcommand{\CO}{{\cal{O}}}
\newcommand{\CP}{{\cal{P}}}

\newcommand{\CS}{{\cal{S}}}
\newcommand{\CT}{{\cal{T}}}
\newcommand{\CU}{{\cal{U}}}
\newcommand{\CV}{{\cal{V}}}
\newcommand{\CW}{{\cal{W}}}
\newcommand{\CX}{{\cal{X}}}

\newcommand{\CZ}{{\cal{Z}}}

\newcommand{\FG}{{\mathfrak{G}}}

\newcommand{\FX}{{\mathfrak{X}}}

\newcommand{\LM}{\textbf{M}^{loc}}

\let\setminus\smallsetminus

\newcommand{\ul}[1]{{\underline{#1}}}
\newcommand{\ol}[1]{{\overline{#1}}}
\newcommand{\wh}[1]{{\widehat{#1}}}
\newcommand{\wt}[1]{{\widetilde{#1}}}

\newcommand{\SSS}{S}

\usepackage{ifthen}

\newcommand{\invlim}[1][]{\ifthenelse{\equal{#1}{}}
{\DS \lim_{\longleftarrow}}
{\DS \lim_{\underset{#1}{\longleftarrow}}}
}

\newcommand{\dirlim}[1][]{\ifthenelse{\equal{#1}{}}
{\DS \lim_{\longrightarrow}}
{\DS \lim_{\underset{#1}{\longrightarrow}}}
}

\newcommand{\dbl}{{\mathchoice{\mbox{\rm [\hspace{-0.15em}[}}
                              {\mbox{\rm [\hspace{-0.15em}[}}
                              {\mbox{\scriptsize\rm [\hspace{-0.15em}[}}
                              {\mbox{\tiny\rm [\hspace{-0.15em}[}}}}
\newcommand{\dbr}{{\mathchoice{\mbox{\rm ]\hspace{-0.15em}]}}
                              {\mbox{\rm ]\hspace{-0.15em}]}}
                              {\mbox{\scriptsize\rm ]\hspace{-0.15em}]}}
                              {\mbox{\tiny\rm ]\hspace{-0.15em}]}}}}
\newcommand{\dpl}{{\mathchoice{\mbox{\rm (\hspace{-0.15em}(}}
                              {\mbox{\rm (\hspace{-0.15em}(}}
                              {\mbox{\scriptsize\rm (\hspace{-0.15em}(}}
                              {\mbox{\tiny\rm (\hspace{-0.15em}(}}}}
\newcommand{\dpr}{{\mathchoice{\mbox{\rm )\hspace{-0.15em})}}
                              {\mbox{\rm )\hspace{-0.15em})}}
                              {\mbox{\scriptsize\rm )\hspace{-0.15em})}}
                              {\mbox{\tiny\rm )\hspace{-0.15em})}}}}

\newcommand{\dotBD}{\vbox{\hbox{\kern2pt\bf.}\vskip-4.5pt\hbox{$\BD$}}}

\DeclareMathOperator{\QIsog}{QIsog}
\DeclareMathOperator{\Nilp}{\CN \!{\it ilp}}

\DeclareMathOperator{\Sets}{\CS \!{\it ets}}

\def\ulHZ{{\underline{\hat Z\!}\,}{}}

\def\longto{\longrightarrow}
\def\isoto{\stackrel{}{\mbox{\hspace{1mm}\raisebox{+1.4mm}{$\SC\sim$}\hspace{-3.5mm}$\longrightarrow$}}}

\newbox\mybox
\def\arrover#1{\mathrel{
       \setbox\mybox=\hbox spread 1.4em{\hfil$\scriptstyle#1$\hfil}
       \vbox{\offinterlineskip\copy\mybox
             \hbox to\wd\mybox{\rightarrowfill}}}}

\newcommand{\BaseOfD}{\BF}

\newcommand{\BaseFldOfLocSht}{k}

\newcommand{\BaseFldInSectUnif}{k}

\newcommand{\genericG}{P}

\DeclareMathOperator{\SpaceFl}{\CF\ell}

\newcommand{\tauLoc}{\hat\tau}

\newcommand{\RemJb}{Remark~4.11}

\begin{document}

\author{Esmail Arasteh Rad\forget{\footnote{Part of this research was carried out while I was visiting Institute For Research In Fundamental Sciences (IPM).}}}
\date{\today}
\title{Local Models For Rapoport-Zink Spaces For Local $\BP$-Shtukas\\}

\maketitle

\begin{abstract}
This article provides a ``local'' complementary to the previous results concerning the local models for the moduli stacks of ``global'' $\FG$-shtukas. Here we study the geometry of Rapoport-Zink spaces for local $\BP$-shtukas by constructing local models for them. We further discuss certain applications, including some results related to the theory of formal nearby cycles associated to these spaces and the semi-simple trace of Frobenius on the corresponding sheaves.

\noindent
{\it Mathematics Subject Classification (2000)\/}: 
11G09,  
(11G18,  
14L05,  
14M15)  
\end{abstract}

\section{Introduction}

Recall that a Shimura datum $(\BG,X,K)$ consists of\\[1mm]
- a reductive group $\BG$ over the ring of rational numbers $\BQ$, with center $Z$,\\[0.05mm]
- $\BG(\BR)$-conjugacy class $X$ of homomorphisms $\BS\to G_\BR$ for the Deligne torus $\BS$,\\[0.05mm]
- a compact open sub-group $K\subseteq \BG(\BA_f)$,\\[1mm]
subject to certain conditions; for example see \cite{MilneShimura}. Here  $\BA_f$ is the ring of finite adeles. Fix a prime number $p$ and write $K=K_p \cdot K^p$ for compact open subgroups $K_p\subseteq \BG(\BQ_p)$ and $K^p\subseteq \BG(\BA^p)$, where $\BA^p$ denotes the ring of adeles away from $p$. The above tuple determines a reflex field $E:=E(\BG,X)$, and the corresponding Shimura variety 
$$
Sh_K(\BG,X)=\BG(\BQ)\backslash\bigl(X\times \BG(\BA_f)/K\bigr).
$$
\noindent  
The Shimura variety $Sh_K(\BG,X)$ admits a canonical integral model $\scrS_K$ over \forget{$W= W(\ol\BF_p)$} $\CO_E$\forget{ for $K=G(\BZ_p)K^p$}, for a sufficiently small $K^p\subseteq \BG(\BA^p)$; see \cite{Kis}. The significance of these varieties come from the fact that they come equipped with many symmetries, which encode important arithmetic data. In addition, for wide range of cases, they appear as moduli spaces for motives, according to Deligne's conception of Shimura varieties \cite{Deligne1} and \cite{Deligne2}. From this perspective, it is expected that the Langlands correspondence will be realized on their cohomology.\\[1mm]
Shimura varieties have local counterparts, which are called \emph{local Shimura varieties}. Recall that a local Shimura variety $Sh(G,[b],\{\mu\})$ is expected to arise from  \emph{local Shimura datum} $(G,\{\mu\},[b])$, according to a conjecture of Rapoport and Viehmann \cite[Subsection 5.2]{R-V}, consisting of \\[1mm]
- a connected reductive group $G$ over $\BQ_p$,\\[0.05mm]
- a (geometric) conjugacy class $\{\mu\}$ of a (minuscule) cocharacter $\mu : \BG_m \to G$,\\[0.05mm]
- a class $[b]$ in $B(G, \mu)$ of Kottwitz set of $\sigma$-conjugacy classes.\\[1mm]
Moreover, they are supposed in addition to be subject to certain conditions, see \cite[Properties 5.4]{R-V}. Note however that the problem with the lake of uniqueness restricts one to the local Shimura data of Hodge type, which means that the local Shimura datum can be embedded in a local Shimura datum of the form $(GL_n,\{\mu'\},[b'])$. \\[1mm]
In this context the integral models for Shimura varieties have local counterparts called \emph{Rapoport-Zink spaces}. This means in particular that their $\ell$-adic cohomology is supposed to eventually realize the local Langlands correspondence, according to a conjecture of Kottwitz \cite{Rap94}. Let us explain it a bit further. Consider an (integral) Shimura datum $(\CP, \{\mu\}, [b])$, where $\CP$ is a smooth affine group scheme over $\BZ_p$ with generic fiber $G$, and set $\CO=\CO_{E}$, for the corresponding reflex field $E:=E_\mu$, i.e. the field of definition of the cocharacter $\mu$. Assume further that $G$ splits over a tamely ramified extension and $\CP$ is parahoric (i.e. the special fiber $\CP_s=\CP_{\BF_p}$ is connected). To such a datum one associates a formal scheme $\breve{\CM}:=\breve{\CM}(\CP, [b], \{\mu\})$ over $\CO$, which is called \emph{Rapoport-Zink space} associated to the tuple $(\CP, [b], \{\mu\})$. The underlying scheme $\breve{\CM}_{\red}$ is a union of affine Deligne-Lusztig varieties. For details see \cite{RV} and \cite{SchWei}. See also \cite{Kim} and Shen \cite{Shen} for generalizations to the Rapoport-Zink spaces of Hodge type and abelian type, respectively.

In \cite{RZ} the authors propose a local model theory to study local properties of Rapoport-Zink spaces. Recall that although the flatness of Rapoport-Zink spaces had been expected by Rapoport and Zink in \cite{RZ}, it was later observed that the flatness might fail in general. The local model in \cite{RZ} is denoted by $\textbf{M}^{naive}$. To achieve the flatness, one should apply certain modifications which leads to the construction of the local model $\LM$ inside $\textbf{M}^{naive}$. 

\noindent
Let $F/\BQ_p$ be a finite extension. A local model triple $(G,\{\mu\},\CP)$ over $F$ consists of\\[1mm]
- a reductive group $G$ over $F$,\\[0.05mm]
- a (geometric) conjugacy class $\{\mu\}$ of a cocharacter $\mu$  of $G$ and\\[0.05mm]
- a parahoric group scheme $\CP$ over $\CO_F$.\\[1mm]
Assuming that $G$ splits over a tamely ramified extension and that the order of the fundamental group $\pi^1(G(\ol\BQ_p
)^{der})$ is prime to $p$, one can associate to a local model datum $(G,\{\mu\},\CP)$, a variety $\scrA:=\scrA(G,\{\mu\},\CP)$ over $k:=\ol\kappa_E$ inside the affine flag variety $\CF\ell_{\CP,k}$, with an action of $\CP\otimes_{\CO_F} k$. The variety $\scrA$ is the union $\cup_{\omega\in Adm_{\CP}(\mu)} S(\omega)$ of affine Schubert varieties $S(\omega)$. Here $Adm_{\CP}(\mu)=\{\omega\in \wt{W}^\CP\backslash\wt{W}/\wt{W}^\CP; \omega \preceq t^\mu  \}$; see remark \ref{RemAffSchVarandIwahori_Weylgp}. A local model $\LM:=\LM(G,\{\mu\},\CP)$ attached to a local model datum $(G,\{\mu\},\CP)$ is a projective scheme $\LM$ over $\CO_E$, with generic fiber $\textbf{M}_\eta^{loc}$ (resp. special fiber $\textbf{M}_s^{loc}$), with an action of $\CP\otimes_{\CO_F}\CO_E$, subject to the following conditions, namely it is flat over $\CO_E$ with reduced special fiber, there is a $\CP\otimes k$-equivariant isomrphism $\textbf{M}_s^{loc}\cong \scrA$ and finally
there is a $G_E$-equivariant isomrphism $\textbf{M}_\eta^{loc}\cong G_E/P_{\{\mu\}}$. Note that in particular all irreducible components of $\LM\otimes k$ are normal and Cohen-Macaulay, see \cite[Theorem~8.4]{PR2}. The existence and uniqueness of $\LM$ is known for $EL$ and $PEL$ cases. In general only the existence is known according to \cite{PZ} and uniqueness is not known; also compare \cite[Proposition~18.3.1]{SchWei}, where the authors use an alternative definition and then they serve uniqueness but not the existence.\\
\noindent
According to the local model theory there is 
a local model roof

\begin{equation}
\xygraph{
!{<0cm,0cm>;<1cm,0cm>:<0cm,1cm>::}
!{(0,0) }*+{\wt{\CM}}="a"
!{(-1.5,-1.5) }*+{\breve{\CM}}="b"
!{(1.5,-1.5) }*+{\LM,}="c"
"a":_{\pi}"b" "a":^{\pi^{loc}}"c"
}  
\end{equation}
\noindent
where $\pi:\wt{\CM}\to\breve{\CM}$ is a $\CP$-torsor and $\pi^{loc}$ is formally smooth of relative dimension $\dim \CP$. In particular for every $x\in\breve{\CM}(k)$ there is a $y\in\LM(k)$ with $\wh{\CO}_{\breve{\CM},x}\cong\wh{\CO}_{\LM,y}$.

\bigskip

In a series of articles \cite{AH_Local}, \cite{AH_Global}, \cite{AH_LR}, \cite{A_CMot} and \cite{AH_LM}, we developed different aspects of the analogues picture over function fields. This includes construction of the moduli stacks of global $\FG$-shtukas and Rapoport-Zink spaces for local $\BP$-shtukas in ramified case, and describing their deformation theory and uniformization theory, see \cite{AH_Global} and \cite{AH_Local}. Moreover we described them as a moduli for motives \cite{A_CMot}, and proved Langlands-Rapoport conjecture for these moduli stacks \cite{AH_LR}. In addition, we established local model theory for the moduli of global $\FG$-shtukas in \cite{AH_LM}. To improve this picture, in the present article we discuss the local geometry of the Rapoport-Zink spaces for local $\BP$-shtukas by constructing local models for them. Namely, we prove the following  

\begin{theorem}
Let $(\BP,\hat{Z},b)$ be a local $\nabla\scrH$-datum and let $\breve{\CM}_{\ul\BL}^{\hat{Z}}:=\breve{\CM}(\BP,\hat{Z},b)$ denote the associated Rapoport-Zink space. There is a roof of morphisms
\begin{equation}
\xygraph{
!{<0cm,0cm>;<1cm,0cm>:<0cm,1cm>::}
!{(0,0) }*+{\wt{\CM}_{\ul\BL}^{\hat{Z}}}="a"
!{(-1.5,-1.5) }*+{\breve{\CM}_{\ul\BL}^{\hat{Z}}}="b"
!{(1.5,-1.5) }*+{\hat{Z},}="c"
"a":_{\pi}"b" "a":^{\pi^{loc}}"c"
}  
\end{equation}

satisfying the following properties

\begin{enumerate}
\item
the morphism $\pi^{loc}$ is formally smooth and
\item
the $L^+\BP$-torsor $\pi: \wt{\CM}_{\ul\BL}^{\hat{Z}}\to \breve{\CM}_{\ul\BL}^{\hat{Z}}$ admits a section $s'$ locally for the \'etale topology on $\breve{\CM}_{\ul\BL}^{\hat{Z}}$ such that $\pi^{loc}\circ s'$ is formally \'etale.
\end{enumerate}

\end{theorem}

This is theorem \ref{ThmRapoportZinkLocalModel} in the text. See also definitions \ref{DefLocalNablaHdata} and \ref{DefRZForLocPSht}, and the assignment (\ref{eqAssignRZToLD}), for the definition of local $\nabla\scrH$-data (which are function fields analogs for local Shimura data), and the associated Rapoport-Zink spaces. Note that a significant part of local model theorems for Shimura varieties and Rapoport-Zink spaces is the construction of the local model, e.g. see \cite{PR1}, \cite{PZ}, \cite{PRS}, and also \cite{Lev}. These constructions are quite involved, especially in the non-PEL cases. But in the function fields setting, the local model is indeed given as a part of the local $\nabla\scrH$-datum, which is the analog of (integral) local Shimura datum. \\  

From the above perspective, this article provides a ``local'' complementary to \cite{AH_LM}, where the authors established the theory of local models for moduli of global $\FG$-shtukas, both in the sense of Beilinson-Drinfeld-Gaitsgory-Varshavsky, and also in the sense of Rapoport and Zink, in the following general setup. Namely, in \cite{AH_LM}, the authors treat the case where $\FG$ is a smooth affine group scheme over a smooth projective curve $C$ over $\BF_q$. Here we also treat a very general case\forget{, namely, we only assume that $\BP$ is a smooth affine group scheme over $\BD$. It can be easily seen that this assumption can not be weakened further}. Like the Shimura variety side, the local model theorem for Rapoport-Zink spaces for local $\BP$-shtukas has several immediate consequences. We discuss some of the applications in the remaining sections. For example it clarifies type of singularities in certain cases. It also helps to answer the questions about the flatness of Rapoport-Zink spaces for local $\BP$-shtukas over their reflex rings. Note that apart from local consequences of the local model theorem, it has also some global consequences. This is because the local model diagram is defined globally. We discuss some of these applications in \cite{AH_MR}.  Moreover, the local model theory can be implemented to study formal nearby cycles cohomology and the semi-simple trace of Frobenius on the cohomology of these spaces by relating them to the better understood semi-simple trace on the usual nearby cycles sheaves corresponding to the boundedness conditions. We discuss this in subsection \ref{Subsect Semi-simple trace and Frob}. In this subsection we first discuss the case of global moduli stacks for $\FG$-shtukas and then we discuss the formal nearby cycles for Rapoport-Zink spaces for local $\BP$-shtukas. Note that, to this goal we crucially use the uniformization theory of global $\FG$-shtukas established in \cite{AH_Global}. For the discussion on the semi-simple trace of Frobenius on the cohomology of (certain) Schubert varieties (here given by our local boundedness conditions) inside twisted affine flag varieties, we refer the reader to \cite{HR1}. Note further that for certain technical reasons we prefer to use a variant  
of Berkovich's formal nearby cycles \cite{Berk}, which has been constructed and studied by Mieda \cite{Mieda}, see Def-Rem \ref{Def-RemMFNC} and Remark \ref{RemBerNVCandRZSpaces}. We show the corresponding nearby cycle sheaves by $R\Psi_{\FX}^\textbf{Ber}$ and $R\Psi_{\FX} \Lambda$, respectively. In particular we prove the following.

\begin{proposition}
 Assume that $(\ul\BL_i)_i:=\hat{\ul\Gamma}(\ul\CG_0)$, for some global $\FG$-shtuka $\ul\CG_0$, where $\hat{\ul\Gamma}(-)$ denote the global-local functor; see \cite[Definition 5.1]{AH_Global}. Here $\FG$ is a parahoric Bruhat-Tits group scheme over $C$, with $\BP_i=\FG_{\nu_i}:=\FG\times_C\wh\CO_{C,\nu_i}$.  For any tuple $\hat{\ul Z}:=(\hat Z_i)_{i=1,\dots,n}$ of boundedness conditions, there is a canonical isomorphism

$$
\pi_{red}^\ast R\Psi_{ \left(\Gamma\backslash\breve{\CM}_{(\ul{\BL}_i)_i}^{\hat{\ul Z}}\right)_\Delta}^\textbf{Ber} \Lambda \to R\Psi_{\left(\breve{\CM}_{(\ul{\BL}_i)_i}^{\hat{\ul Z}}\right)_\Delta}\Lambda,
$$
for a separated discrete subgroup $\Gamma\subseteq J_{(\BL_i)_i}$ (see paragraph \ref{Point1}). Here $\pi$ denotes the projection $\breve{\CM}_{(\ul{\BL}_i)_i}^{\hat{\ul Z}}\to\Gamma\backslash\breve{\CM}_{(\ul{\BL}_i)_i}^{\hat{\ul Z}}$, and the subscript $\Delta$ indicates that these spaces are obtained by pulling back the corresponding spaces under the morphism $\Spf\BaseFldInSectUnif\dbl\xi\dbr\to\Spf\BaseFldInSectUnif\dbl\ul\xi\dbr$, given by $\xi_i \mapsto \xi$.\\
\end{proposition}
This is Proposition \ref{PropMiedaCompBerkovich} in the text. Moreover as a consequence we observe the following in Corollary \ref{Cor_SemisimpleTrace}.

\begin{corollary}
Let $\nu$ be a place on $C$ and set $\BP:=\BP_\nu$. Let $\ul\BL$ be a local $\BP$-shtuka, which comes from a global $\FG$-shtuka $\ul\CG$ under the functor $\hat{\Gamma}_{\nu}(-)$. Set $\breve{\CM}:=\breve{\CM}_{\ul\BL}^{\hat{Z}}$ and $\kappa=\kappa_{\hat{Z}}$. Let $\kappa_r/\kappa$ be a finite extension of degree $r$. Let $x$ be a point in $\breve{\CM}(\kappa_r)$ and let $y$ be the image $\pi^{loc}(y')$ of a point $y'$ in $\wt\CM_{\ul\BL}^{\hat{Z}}$ above $x$ under $\pi$ (see the local model roof (\ref{EqnablaHRoof})). Then 
$$
tr^{ss}\left(Frob_r; \left(R\Psi_{\breve{\CM},c}\ol\BQ_\ell\right)_x\right)=tr^{ss}\left(Frob_r; \left(R\psi_{\hat{Z}}\ol\BQ_\ell\right)_y\right).
$$
Here $Frob_r$ denotes the geometric Frobenius in $\Gal(\ol\kappa_r/\kappa_r)$.

\end{corollary}
\section*{Acknowledgment}\label{Acknowledgment}

I warmly thank Urs Hartl for numerous stimulating conversations and explanations, and also for his continues encouragements and support. I thank Peter Scholze for inspiring conversations in the sideline of a 2021 MFO workshop, as well as his valuable comments \forget{regarding formal nearby cycles cohomology with compact support and the semi-simple trace of Frobenius,} that helped with reformulating and proving a result related to formal nearby cycles. I would also like to thank Yoichi Mieda for his explanations related to his theory of formal nearby cycles and for his valuable comments. In addition, I would like to thank Somayeh Habibi, Mohsen Asgharzadeh and Mohammad Hadi Hedayatzadeh for their comments, as well as inspiring conversations. I profited a lot from lectures given by Sophie Morel at IPM and wish to express my deep appreciation. This paper grew out of lectures delivered at a number theory conference held at IPM (April-May 2017) and for that I would like to thank the organizers and the staff.

\tableofcontents 

\subsection{Notation and conventions}\label{SectNotation and Conventions}

Throughout this article we denote by
\begin{tabbing}

$\genericG_\nu:=\FG\times_C\Spec Q_\nu,$\; \=\kill

$\BF_q$\> a finite field with $q$ elements of characteristic $p$,\\[1mm]
$C$\> a smooth projective geometrically irreducible curve over $\BF_q$,\\[1mm]
$Q:=\BF_q(C)$\> the function field of $C$,\\[1mm]







$\BaseOfD$\> a finite field containing $\BF_q$,\\[1mm]

$\wh A:=\BF\dbl z\dbr$\>  the ring of formal power series in $z$ with coefficients in $\BF$ ,\\[1mm]
$\wh Q:=\Quot(\wh A)$\> its fraction field,\\[1mm]

$\nu$\> a closed point of $C$, also called a \emph{place} of $C$,\\[1mm]
$\BF_\nu$\> the residue field at the place $\nu$ on $C$,\\[1mm]

$A_\nu$\> the completion of the stalk $\CO_{C,\nu}$ at $\nu$,\\[1mm]

$Q_\nu:=\Quot(A_\nu)$\> its fraction field,\\[1mm]

$\BD_R:=\Spec R\dbl z \dbr$ \> \parbox[t]{0.79\textwidth}{\Verkuerzung
{
the spectrum of the ring of formal power series in $z$ with coefficients in an $\BaseOfD$-algebra $R$,
}

{}
}

\Verkuerzung
{
\\[1mm]
$\hat{\BD}_R:=\Spf R\dbl z \dbr$ \> the formal spectrum of $R\dbl z\dbr$ with respect to the $z$-adic topology.
}
{}
\end{tabbing}

\noindent
For a formal scheme $\wh S$ we denote by $\Nilp_{\wh S}$ the category of schemes over $\wh S$ on which an ideal of definition of $\wh S$ is locally nilpotent. We  equip $\Nilp_{\wh S}$ with the \'etale topology. We also denote by
\begin{tabbing}
$\genericG_\nu:=\FG\times_C\Spec \wh Q_\nu,$\; \=\kill

$n\in\BN_{>0}$\> a positive integer,\\[1mm]
$\ul \nu:=(\nu_i)_{i=1\ldots n}$\> an $n$-tuple of closed points of $C$,\\[1mm]
$\BA_C^\ul\nu$\> the ring of rational adeles of $C$ outside $\ul\nu$,\\[1mm]

\forget{
$A_\ul\nu$\> the completion of the local ring $\CO_{C^n,\ul\nu}$ of $C^n$ at the closed point $\ul\nu=(\nu_i)$,\\[1mm]

$\Nilp_{A_\ul\nu}:=\Nilp_{\Spf A_\ul\nu}$\> \parbox[t]{0.79\textwidth}{the category of schemes over $C^n$ on which the ideal defining the closed point $\ul\nu\in C^n$ is locally nilpotent,}\\[2mm]
}
$\Nilp_{\BaseOfD\dbl\zeta\dbr}$\lang{$:=\Nilp_{\hat\BD}$}\> \parbox[t]{0.79\textwidth}{the category of $\BD$-schemes $S$ for which the image of $z$ in $\CO_S$ is locally nilpotent. We denote the image of $z$ by $\zeta$ since we need to distinguish it from $z\in\CO_\BD$.}\\[2mm]
$\FG$\> a smooth affine group scheme of finite type over $C$,
with connected \\ 

\> reductive generic fiber $G$, \\[2mm]
$\BP_\nu:=\FG\times_C\Spec  A_\nu,$ \> the base change of $\FG$ to $\Spec A_\nu$,\\[1mm]
$\genericG_\nu:=\FG\times_C\Spec Q_\nu,$ \> the generic fiber of $\BP_\nu$ over $\Spec Q_\nu$,\\[1mm]

$\BP$\> a smooth affine group scheme of finite type over $\BD=\Spec\BaseOfD\dbl z\dbr$,\\ 
\>  with connected 
 reductive generic fiber $P$ over $\Spec\BaseOfD\dpl z\dpr$.\\[1mm] 
\end{tabbing}

We denote by $\sigma_S \colon  S \to S$ the $\BF_q$-Frobenius endomorphism which acts as the identity on the points of $S$ and as the $q$-power map on the structure sheaf. Likewise we let $\hat{\sigma}_S\colon S\to S$ be the $\BaseOfD$-Frobenius endomorphism of an $\BaseOfD$-scheme $S$. \forget{We set
\begin{tabbing}
$\genericG_\nu:=\FG\times_C\Spec Q_\nu,$\; \=\kill
$C_S := C \times_{\Spec\BF_q} S$ ,\> and \\[1mm]
$\sigma := \id_C \times \sigma_S$.

\end{tabbing}
}

\begin{remark}\label{RemAffSchVarandIwahori_Weylgp}

 Assume that the generic fiber $\genericG$ of $\BP$ over $\Spec\BaseOfD\dpl z\dpr$ is connected reductive. Consider the base change $\genericG_L$ of $\genericG$ to $L=\BaseOfD^\alg\dpl z\dpr$. Let $S$ be a maximal split torus in $\genericG_L$ and let $T$ be its centralizer. Since $\BaseOfD^\alg$ is algebraically closed, $\genericG_L$ is quasi-split and so $T$ is a maximal torus in $\genericG_L$. Let $N = N(T)$ be the normalizer of $T$ and let $\CT^0$ be the identity component of the N\'eron model of $T$ over $\CO_L=\BaseOfD^\alg\dbl z\dbr$. The \emph{Iwahori-Weyl group} associated with $S$ is the quotient group $\wt{W}= N(L)\slash\CT^0(\CO_L)$. It is an extension of the finite Weyl group $W_0 = N(L)/T(L)$ by the coinvariants $X_\ast(T)_I$ under $I=\Gal(L^\sep/L)$:
$$
0 \to X_\ast(T)_I \to \wt W \to W_0 \to 1.
$$
By \cite[Proposition~8]{H-R} there is a bijection
\begin{equation}\label{EqSchubertCell}
L^+\BP(\BaseOfD^\alg)\backslash L\genericG(\BaseOfD^\alg)/L^+\BP(\BaseOfD^\alg) \isoto \wt{W}^\BP  \backslash \wt{W}\slash \wt{W}^\BP
\end{equation}
where $\wt{W}^\BP := (N(L)\cap \BP(\CO_L))\slash \CT^0(\CO_L)$, and where $LP(R)=P(R\dbl z\dbr)$ and $L^+\BP(R)=\BP(R\dbl z\dbr)$ are the loop group, resp.\ the group of positive loops of $\BP$; see \cite[\S\,1.a]{PR2}, or \cite[\S4.5]{B-D}, \cite{Ngo-Polo} and \cite{Faltings03} when $\BP$ is constant. Let $\omega\in \wt{W}^\BP\backslash \wt{W}/\wt{W}^\BP$ and let $\BaseOfD_\omega$ be the fixed field in $\BaseOfD^\alg$ of $\{\gamma\in\Gal(\BaseOfD^\alg/\BaseOfD)\colon \gamma(\omega)=\omega\}$. There is a representative $g_\omega\in L\genericG(\BaseOfD_\omega)$ of $\omega$; see \cite[Example~4.12]{AH_Local}. The \emph{Schubert variety} $\CS(\omega)$ associated with $\omega$ is the ind-scheme theoretic closure of the $L^+\BP$-orbit of $g_\omega$ in $\SpaceFl_\BP\whtimes_{\BaseOfD}\BaseOfD_\omega$. It is a reduced projective variety over $\BaseOfD_\omega$. For further details see \cite{PR2} and \cite{Richarz}.  

\end{remark}

\section{Rapoport-Zink spaces for local $\BP$-shtukas}\label{SectRZSpaces}

In this section we first recall the definition of local $\BP$-shtukas and some of their basic properties, see subsection \ref{SubSectLocalPSht} below. Then, in subsection \ref{SubSecLocalData}, we define local $\nabla\scrH$-data, and explain how one assigns a Rapoport-Zink space to a local $\nabla\scrH$-datum.

\subsection{Local $\BP$-shtukas}\label{SubSectLocalPSht}

Let $\BaseOfD$ be a finite field and $\BaseOfD\dbl z\dbr$ be the power series ring over $\BaseOfD$ in the variable $z$. We let $\BP$ be a smooth affine group scheme over $\BD:=\Spec\BaseOfD\dbl z\dbr$ with connected fibers. Set $\dot\BD:=\Spec \BF\dpl z\dpr$.

\begin{definition}\label{DefLoopGps}
The \emph{group of positive loops associated with $\BP$} is the infinite dimensional affine group scheme $L^+\BP$ over $\BaseOfD$ whose $R$-valued points for an $\BaseOfD$-algebra $R$ are $L^+\BP(R):=\BP(R\dbl z\dbr):=\BP(\BD_R):=\Hom_\BD(\BD_R,\BP)\,$.
The \emph{group of loops associated with $P$} is the $\fpqc$-sheaf of groups $LP$ over $\BaseOfD$ whose $R$-valued points for an $\BaseOfD$-algebra $R$ are $LP(R):=P(R\dpl z\dpr):=P(\dot{\BD}_R):=\Hom_{\dot\BD}(\dot\BD_R,P)\,,$ where we write $R\dpl z\dpr:=R\dbl z \dbr[\frac{1}{z}]$ and $\dot{\BD}_R:=\Spec R\dpl z\dpr$. It is representable by an ind-scheme of ind-finite type over $\BaseOfD$; see \cite[\S\,1.a]{PR2}, or \cite[\S4.5]{B-D}, \cite{Ngo-Polo}, \cite{Faltings03} when $\BP$ is constant.
Let $\scrH^1(\Spec \BaseOfD,L^+\BP)\,:=\,[\Spec \BaseOfD/L^+\BP]$ (respectively $\scrH^1(\Spec \BaseOfD,LP)\,:=\,[\Spec \BaseOfD/LP]$) denote the classifying space of $L^+\BP$-torsors (respectively $LP$-torsors). It is a stack fibered in groupoids over the category of $\BaseOfD$-schemes $S$, whose category $\scrH^1(\Spec \BaseOfD,L^+\BP)(S)$ consists of all $L^+\BP$-torsors (resp.\ $LP$-torsors) on $S$. The inclusion of sheaves $L^+\BP\subset LP$ gives rise to the natural 1-morphism 
\begin{equation}\label{EqLoopTorsor}
\scrH^1(\Spec \BaseOfD,L^+\BP)\longto \scrH^1(\Spec \BaseOfD,LP),~\CL_+\mapsto \CL\,.
\end{equation}
\end{definition}

\begin{definition}
The affine flag variety $\SpaceFl_\BP$ is defined to be the ind-scheme representing the $fpqc$-sheaf associated with the presheaf defined by sending $
R$ to $L\genericG(R)/L^+\BP(R)\;=\;P\left(R\dpl z \dpr \right)/\BP \left(R\dbl z\dbr\right)$, on the category of $\BaseOfD$-algebras; compare Definition~\ref{DefLoopGps}.
\end{definition}

\begin{remark}\label{RemFlagisquasiproj}
 Note that $\SpaceFl_\BP$ is ind-quasi-projective over $\BaseOfD$ according to Pappas and Rapoport \cite[Theorem~1.4]{PR2}, and hence ind-separated and of ind-finite type over $\BaseOfD$. The quotient morphism $LP \to \SpaceFl_\BP$ admits sections locally for the \'etale topology.\forget{They proceed as follows. When $\BP = \SL_{r,\BD}$, the \fpqc-sheaf $\CF\ell_\BP$ is called the \emph{affine Grassmanian}. It is an inductive limit of projective schemes over $\BF$, that is, ind-projective over $\BF$; see \cite[Theorem~4.5.1]{B-D} or \cite{Faltings03, Ngo-Polo}.  By \cite[Proposition~1.3]{PR2} and \cite[Proposition~2.1]{AH_Global} there is a faithful representation $\BP ֒\to \SL_r$ with quasi-affine quotient. Pappas and Rapoport show in the proof of \cite[Theorem~1.4]{PR2} that $\CF\ell_\BP \to \CF\ell_{\SL_r}$ is a locally closed embedding, and moreover, if $\SL_r /\BP$ is affine, then $\CF\ell_\BP \to \CF\ell_{\SL_r}$ is even a closed embedding and $\CF\ell_\BP$ is ind-projective.} Moreover, if the fibers of $\BP$ over $\BD$ are geometrically connected, then $\CF\ell_\BP$ is ind-projective if and only if $\BP$ is a parahoric group scheme in the sense of Bruhat and Tits \cite[D\'efinition 5.2.6]{B-T}; see \cite[Theorem A]{Richarz16}. 

\end{remark}

\noindent
Recall that for p-divisible groups $X$ and $Y$, one defines the following

\begin{enumerate}
\item
An \emph{isogeny} $f:X\to Y$ is a morphism which is an epimorphism as \fppf-sheaves and whose kernel is representable by a finite flat group scheme over $S$.
\item
A \emph{quasi-isogeny} is a global section $f$ of the Zariski sheaf $\Hom(X,Y)\otimes_\BZ \BQ$  such that $n\cdot f$ is an isogeny locally on $S$, for an integer $n\in\BZ$. 
\end{enumerate}

\noindent
Let us now recall the analogous definition over function fields, where we instead have \emph{local $\BP$-shtukas} and \emph{quasi-isogenies} between them.

\begin{definition}\label{localSht}
\begin{enumerate}
\item
A local $\BP$-shtuka over $S\in \Nilp_{\BaseOfD\dbl\zeta\dbr}$ is a pair $\ul \CL = (\CL_+,\tau)$ consisting of an $L^+\BP$-torsor $\CL_+$ on $S$ and an isomorphism of the associated loop group torsors $\tauLoc\colon  \hat{\sigma}^\ast \CL \to\CL$. 

\item
A \emph{quasi-isogeny} $f\colon\ul\CL\to\ul\CL'$ between two local $\BP$-shtukas $\ul{\CL}:=(\CL_+,\tau)$ and $\ul{\CL}':=(\CL_+' ,\tau')$ over $S$ is an isomorphism of the associated $LP$-torsors $f \colon  \CL \to \CL'$ such that the following diagram  

\[
\xymatrix {
\hat{\sigma}^\ast\CL \ar[r]^{\tau} \ar[d]_{\hat{\sigma}^\ast f} & \CL\ar[d]^f  \\
\hat{\sigma}^\ast\CL' \ar[r]^{\tau'}&  \CL' \;.
}
\]

becomes commutative.
\item
We denote by $\QIsog_S(\ul{\CL},\ul{\CL}')$ the set of quasi-isogenies between $\ul{\CL}$ and $\ul{\CL}'$ over $S$. 

\item
We let $Loc-\BP-Sht(S)$ denote the category of local $\BP$-shtukas over $S$ with quasi-isogenies as the set of morphisms. 

\end{enumerate}
\end{definition}

Recall that quasi-isogenies of p-divisible groups are rigid in the following sense. Let $X$ and $Y$ be p-divisible groups over $S$. Let $\ol S\to S$ be a nilpotent thickening, i.e. a closed immersion defined by a nilpotent sheaf of ideal. Then, the restriction $QIsog_S(X,Y)\to QIsog_{\ol S}(\ol X,\ol Y)$ between the set of quasi-isogenies is a bijection. Likewise, local $\BP$-shtukas enjoy a similar rigidity property.

\bigskip

\begin{proposition}[Rigidity of quasi-isogenies for local $\BP$-shtukas] \label{PropRigidityLocal}
Let $S$ be a scheme in $\Nilp_{\BaseOfD\dbl\zeta\dbr}$ and
let $j \colon  \bar{S}\rightarrow S$ be a closed immersion defined by a sheaf of ideals $\CI$ which is locally nilpotent.
Let $\ul{\CL}$ and $\ul{\CL}'$ be two local $\BP$-shtukas over $S$. Then the map
$\QIsog_S(\ul{\CL}, \ul{\CL}') \longto \QIsog_{\bar{S}}(j^*\ul{\CL}, j^*\ul{\CL}')$, defined by sending $\quad f$ to $j^*f$, is a bijection of sets.
\end{proposition}

\begin{proof} See \cite[Proposition 2.11]{AH_Local}.
\end{proof}

\subsection{Local shtuka data and the corresponding Rapoport-Zink spaces}\label{SubSecLocalData}

\noindent
Fix an algebraic closure $\BaseOfD\dpl\zeta\dpr^\alg$ of $\BaseOfD\dpl\zeta\dpr$. For a finite extensions of discrete valuation rings $R/\BaseOfD\dbl\zeta\dbr$ with $R\subset\BaseOfD\dpl\zeta\dpr^\alg$, we denote by $\kappa_R$ its residue field, and we let $\Nilp_R$ be the category of $R$-schemes on which $\zeta$ is locally nilpotent. We also set $\wh{\SpaceFl}_{\BP,R}:=\SpaceFl_\BP\whtimes_{\BaseOfD}\Spf R$ and $\wh{\SpaceFl}_\BP:=\wh{\SpaceFl}_{\BP,\BaseOfD\dbl\zeta\dbr}$. Before we establish the assignment of a Rapoport-Zink space to a local $\nabla\scrH$-datum, we should recall that $\wh{\SpaceFl}_{\BP,R}$ can be viewed as an unbounded Rapoport-Zink space for local $\BP$-shtukas. We will explain this in Proposition \ref{PropUnbdRZFunctor} below. For now, let us consider the following functor
\begin{eqnarray}\label{eqPAFlag}
\ul\CM:&(\Nilp_R)^o &\longto  \Sets  \hspace{7cm}\vspace{-2mm}\\ \nonumber
&\SSS &\longmapsto  \big\{\text{Isomorphism classes of }(\CL_+,\delta);\text{where: }\\ \nonumber  
& &~~~~~~~~~~~~~~~~~~~~~~~~~~~~-\CL_+~\text{is an $L^+\BP$-torsor over $\SSS$ and}\\ \nonumber
& &  ~~~~~~~~~~~~~~~~~~~~~~~~~~~~ -\text{a trivialization $\delta\colon  \CL \to LP_S$ of the}\\ \nonumber
& & ~~~~~~~~~~~~~~~~~~~~~~~~~~~~~~~~~\text{associated loop torsors}
\big\}. 
\end{eqnarray}

\begin{proposition}\label{PropFlRepUnBoundedRZ}
The ind-scheme $\wh{\CF\ell}_{\BP,R}$ pro-represents the above functor.

\end{proposition}

\begin{proof}
In order to illustrate how the representablity works, here we briefly sketch the proof, and we refer the reader to \cite[Theorem~4.4.]{AH_Local} for further details. We assume that $R=\BF\dbl\zeta\dbr$. Consider a pair $(\CL_+,\delta)\in \ul\CM(S)$. Choose an \fppf-covering $S' \to S$ which trivializes $\CL_+$, then the morphism $\delta$ is given by an element $g' \in LP(S')$. The image of the element $g' \in LP(S')$ under  $LP(S')\to\wh{\CF\ell}_\BP (S')$ is independent of the choice of the trivialization, and since $(\CL_+,\delta)$ is defined over $S$, it descends to a point $x \in \wh{\CF\ell}_\BP$. 

Conversely let $x$ be in $\wh{\CF\ell}_\BP(S)$, for a scheme $S\in \Nilp_{\BF\dbl\zeta\dbr}$. The projection morphism $LP \to \CF\ell_\BP$ admits local sections for the \'etale topology by \cite[Theorem~1.4]{PR2}. Hence over an \'etale covering $S' \to S$ the point $x$ can be represented by an element $g'\in LP(S')$. We let $(\CL_+',\delta')=((L^+\BP)_{S'},g')$. It can be shown that it descends and gives $(\CL_+,\delta)$ over $S$.
\end{proof}

\noindent
Here we recall the definition of local boundedness condition from \cite[Definition~4.8]{AH_Local}.

\begin{definition}\label{DefLBC}
\begin{enumerate}
\item\label{DefBDLocal_A}
 For a finite extension of discrete valuation rings $\BaseOfD\dbl\zeta\dbr\subset R\subset\BaseOfD\dpl\zeta\dpr^\alg$ we consider closed ind-subschemes $\hat{Z}_R\subset\wh{\SpaceFl}_{\BP,R}$. We call two closed ind-subschemes $\hat{Z}_R\subset\wh{\SpaceFl}_{\BP,R}$ and $\hat{Z}'_{R'}\subset\wh{\SpaceFl}_{\BP,R'}$ \emph{equivalent} if there is a finite extension of discrete valuation rings $\BaseOfD\dbl\zeta\dbr\subset\wt R\subset\BaseOfD\dpl\zeta\dpr^\alg$ containing $R$ and $R'$ such that $\hat{Z}_R\whtimes_{\Spf R}\Spf\wt R \,=\,\hat{Z}'_{R'}\whtimes_{\Spf R'}\Spf\wt R$ as closed ind-subschemes of $\wh{\SpaceFl}_{\BP,\wt R}$.
\item\label{DefBDLocal_C}
We define a (local) \emph{bound} to be an equivalence class $\hat{Z}:=[\hat{Z}_R]$ of closed ind-subschemes $\hat{Z}_R\subset\wh{\SpaceFl}_{\BP,R}$, such that 
\begin{enumerate}
\item all the ind-subschemes $\hat{Z}_R$ are stable under the left $L^+\BP$-action on $\SpaceFl_\BP$ and
\item
 the special fibers $Z_R:=\hat{Z}_R\whtimes_{\Spf R}\Spec\kappa_R$ are quasi-compact subschemes of the ind scheme $\SpaceFl_\BP\whtimes_{\BaseOfD}\Spec\kappa_R$. 
\end{enumerate}
\item\label{DefBDLocal_B} Let $\hat{Z}=[\hat{Z}_R]$ be an equivalence class in the above sense. The \emph{reflex ring} $R_{\hat{Z}}$ is defined as the intersection of the fixed field of $\{\gamma\in\Aut_{\BaseOfD\dbl\zeta\dbr}(\BaseOfD\dpl\zeta\dpr^\alg)\colon \gamma({\hat{Z}})=\hat{Z}\,\}$ in $\BaseOfD\dpl\zeta\dpr^\alg$ with all the finite extensions $R\subset\BaseOfD\dpl\zeta\dpr^\alg$ of $\BaseOfD\dbl\zeta\dbr$ over which a representative $\hat{Z}_R$ of $\hat{Z}$ exists. 
 
\item \label{DefBDLocal_D}
Let $\hat{Z}$ be a bound with reflex ring $R_{\hat{Z}}$. Let $\CL_+$ and $\CL_+'$ be $L^+\BP$-torsors over a scheme $S$ in $\Nilp_{R_{\hat{Z}}}$ and let $\delta\colon \CL\isoto\CL'$ be an isomorphism of the associated $LP$-torsors. We consider an \'etale covering $S'\to S$ over which trivializations $\alpha\colon\CL_+\isoto(L^+\BP)_{S'}$ and $\alpha'\colon\CL_+'\isoto(L^+\BP)_{S'}$ exist. Then the automorphism $\alpha'\circ\delta\circ\alpha^{-1}$ of $(LG)_{S'}$ corresponds to a morphism $S'\to LG\whtimes_\BaseOfD\Spf R_{\hat{Z}}$. We say that $\delta$ is \emph{bounded by $\hat{Z}$} if for any such trivialization and for all finite extensions $R$ of $\BaseOfD\dbl\zeta\dbr$ over which a representative $\hat{Z}_R$ of $\hat{Z}$ exists the induced morphism $S'\whtimes_{R_{\hat{Z}}}\Spf R\to LP\whtimes_\BaseOfD\Spf R\to \wh{\SpaceFl}_{\BP,R}$ factors through $\hat{Z}_R$. Furthermore we say that a local $\BP$-shtuka $(\CL, \tauLoc)$ is \emph{bounded by $\hat{Z}$} if the isomorphism $\tauLoc^{-1}$ is bounded by ${\hat{Z}}$. Assume that ${\hat{Z}}=\CS(\omega)\whtimes_\BaseOfD\Spf \BaseOfD\dbl\zeta\dbr$ for a \emph{Schubert variety} $\CS(\omega)\subseteq \CF\ell_\BP$ , with $\omega\in \wt{W}$; see \cite{PR2}. Then we say that $\delta$ is \emph{bounded by $\omega$}.

\end{enumerate}
\end{definition}

\begin{remark}\label{RemBCSpecialFiber}
 Note that in the part b) of the above Definition one can observe that $Z_R$ arise by base change from a unique closed subscheme $Z\subset\SpaceFl_\BP\whtimes_\BaseOfD\kappa_{R_{\hat{Z}}}$. This is because the Galois descent for closed subschemes of $\SpaceFl_\BP$ is effective. We call $Z$ the \emph{special fiber} of the bound ${\hat{Z}}$. Note that when $\BP$ is parahoric, it is a projective scheme over $\kappa_{R_{\hat{Z}}}$ by~\cite[Remark 4.3]{AH_Local} and \cite[Lemma~5.4]{H-V}.
 
\end{remark}

\begin{remark}
Note that the condition ii) of the part (b) of the above definition implies that the $\hat{Z}_R$ are formal schemes in the sense of \cite[I$_{\rm new},~ $\S$~10$]{EGA}; see \cite[Remark 4.10]{AH_Local}.
\end{remark}

\begin{remark}
Note that the boundedness condition stated in part \ref{DefBDLocal_D} of the above definition is satisfied for all trivializations and for all such finite extensions $R$ of $\BF_q\dbl\zeta\dbr$ if and only if it is satisfied for one trivialization and for one such finite extension. Namely, first observe that by the $L^
+\BP$-invariance of $\hat{Z}$ the definition is independent of the trivializations. For the fact that
one such extension suffices, see \cite[Remark 4.6]{AH_Local}
\end{remark}

\begin{remark}\label{RemNablaHData}
In our analogous picture over function fields, the Shimura datum $(\BG,X,K)$, would be replaced by tuples $(\FG, \ul{\hat{Z}}, H)$, which we call \emph{$\nabla\scrH$-datum}. A $\nabla\scrH$-datum $(\FG, \ul{\hat{Z}}, H)$ consists of a smooth affine group scheme $\FG$ over a smooth projective curve $C$ over $\BF_q$, an $n$-tuple of (local) bounds $\ul{\hat{Z}}:=({\hat{Z}}_{\nu_i})_{i=1\dots n}$, in the sense of Definition \ref{DefLBC}, at the fixed characteristic places $\nu_i\in C$ and a compact open subgroup $H\subseteq \FG(\BA_C^\ul\nu)$. A morphism $(\FG, \ul{\hat{Z}}, H)\to(\FG', \ul{\hat{Z}}', H')$  between two $\nabla\scrH$-data is a morphism $\rho: \FG\to\FG'$ such that 
\begin{enumerate}

\item
the inclusion ${\hat{Z}}_{\nu_i}\to\wh{\CF\ell}_{\BP_{\nu_i}}$ followed by the induced morphism $\wh{\CF\ell}_{\BP_{\nu_i,R}}\to\wh{\CF\ell}_{\BP_{\nu_i,R}'}$ factors through $\hat{Z}_{\nu_i,R}'$, for a dvr $R\supseteq \BF\dbl \zeta \dbr$ and 
\item
the image of $H$ under the induced morphism $\FG(\BA_C^\ul\nu)\to\FG(\BA_C^\ul\nu)$ lies in $H'$.
\end{enumerate}
To such a datum we associate a moduli stack $\nabla_n^{H,\ul{\hat{Z}}}\scrH^1(C,\FG)^\ul\nu$,  parametrizing global $\FG$-shtukas with level $H$-structure which are in addition bounded by $\ul{\hat{Z}}$, see \cite{AH_Global}, in a functorial way;  see \cite{Paul}. Note that one can define $\nabla\scrH$-data more generally. Namely, replacing the $n$-tuple of (local) bounds $\ul{\hat{Z}}:=({\hat{Z}}_{\nu_i})_{i=1\dots n}$ by a global boundedness condition, allows to define the corresponding moduli stacks over the n-fold product of the reflex curve; see \cite[Definition 3.1.3]{AH_LM}.

\end{remark}

In analogy with the Shimura variety side we define

\begin{definition}\label{DefLocalNablaHdata}
\emph{A local $\nabla\scrH$-datum} is a tuple $(\BP,\hat{Z}, b)$ consisting of\\[1mm]
- a smooth affine group scheme $\BP$ over $\BD$ with connected reductive generic fiber $P$,\\[0.5mm]
- a local bound $\hat{Z}$ in the sense of Definition \ref{DefLBC}, and\\[0.5mm]
- a $\sigma$-conjuagacy class of an element $b\in P(\ol\BF \dpl z \dpr)$. 
\end{definition} 
 
To a local $\nabla\scrH$-datum $(\BP,\hat{Z}, b)$ one may associate a formal scheme $\breve{\CM}(\BP,\hat{Z}, b)$ which is a moduli space for local $\BP$-shtukas together with a quasi-isogeny to a fixed local $\BP$-shtuka $\ul\BL$, determined by the local $\nabla\scrH$-datum. In analogy with number fields, they are called Rapoport-Zink spaces (for local $\BP$-shtukas). These moduli spaces were first introduced and studied in \cite{H-V} for the case where $\BP$ is a constant split reductive group over $\BD$, and then generalized to the case where $\BP$ is a smooth affine group scheme over $\BD$ with connected reductive generic fiber in \cite{AH_Local}. Here we briefly recall the construction of these formal schemes.

Let $\ul\BL$ be a local $\BP$-shtuka over $\BF$. The bound $\hat{Z}$ determines the reflex ring $R_{\hat{Z}}$. Consider the following functor

\begin{eqnarray}\label{eqUBRZ}
\ul{\breve{\CM}}_{\ul\BL}:&(\Nilp_{R_{\hat{Z}}})^o &\longto  \Sets  \hspace{7cm}\vspace{-2mm}\\ \nonumber
&\SSS &\longmapsto  \big\{\text{Isomorphism classes of }(\ul\CL,\delta); \text{where: }\\ \nonumber
& &~~~~~~~~~~~~~~~~~~~~~~~~~~~~-\ul\CL~\text{is a local $\BP$-shtuka over $\SSS$ and}\\ \nonumber
& &  ~~~~~~~~~~~~~~~~~~~~~~~~~~~~ -\text{$\ol\delta\colon  \ul\CL_{\ol S} \to \ul\BL_\ol S$ is a quasi-isogeny}
\big\}.  
\end{eqnarray}
\noindent
Here $\ol S$ is the closed subscheme of $S$ defined by $\zeta=0$.

\begin{proposition}\label{PropUnbdRZFunctor}
For a trivialized local $\BP$-shtuka $\ul\BL$ the above functor $\ul{\breve{\CM}}_{\ul\BL}$ is pro-representable by $\wh\CF\ell_{\BP, R_{\hat{Z}}}$. 
\end{proposition}
 
\begin{proof}
By rigidity of quasi-isogenies, Proposition \ref{PropRigidityLocal}, the quasi-isogeny $\ol\delta\colon  \ul\CL_{\ol S} \to \ul\BL_\ol S$ lifts to a unique quasi-isogeny $\delta: \ul\CL:=(\CL_+,\delta) \to \ul\BL_S$ over $S$, which in particular gives the isomorphism $\delta: \CL\to LP_S$, vice versa an isomorphism $\delta: \CL\to LP_S$ of torsors induces a unique quasi-isogeny $\delta: \ul\CL_\ol S\to \ul\BL_\ol S$. This obviously gives a natural isomorphism of the functors \ref{eqUBRZ} and \ref{eqPAFlag}, and thus the proposition follows from Proposition \ref{PropFlRepUnBoundedRZ}.   
\end{proof} 

\noindent
Consider the following sub-functor of $\ul{\breve{\CM}}_{\ul\BL}$.
 
\noindent
\begin{definition}\label{DefRZForLocPSht}
 Let $\hat{Z}=[\hat{Z}_R]$ be a bound with reflex field $R_{\hat{Z}}$. Define \emph{the Rapoport-Zink space for (bounded) local $\BP$-shtukas}, as the space given by the following functor of points
\begin{eqnarray*}
\ul{\breve{\CM}}_{\ul\BL}^{\hat{Z}}:&(\Nilp_{R_{\hat{Z}}})^o &\longto  \Sets  \hspace{7cm}\vspace{-2mm}\\
&\SSS &\longmapsto  \big\{\text{Isomorphism classes of }(\ul\CL,\delta)\colon\;\text{where: }\\
& &~~~~~~~~~~~~~~~~~~~~~~~~~~~~-\ul\CL~\text{is a local $\BP$-shtuka}\\
& & ~~~~~~~~~~~~~~~~~~~~~~~~~~~~~~~~\text{ over $\SSS$ bounded by ${\hat{Z}}$ and}\\ 
& &  ~~~~~~~~~~~~~~~~~~~~~~~~~~~~ -\text{$\ol\delta\colon  \ul\CL_{\ol S} \to \ul\BL_\ol S$ a quasi-isogeny}
\big\}. 
\end{eqnarray*}

Note that the datum $(\BP,{\hat{Z}}, b)$ determines the reflex ring $R_{\hat{Z}}$, see Definition \ref{DefLBC}, and a local $\BP$-shtuka $\ul\BL:=(L^+\BP,b\hat{\sigma})$. Thus we may establish the following
\begin{equation} \label{eqAssignRZToLD}
(\BP,{\hat{Z}}, b)\mapsto \breve{\CM}(\BP,{\hat{Z}}, b),
\end{equation}
which assigns the Rapoport-Zink space $\breve{\CM}(\BP,{\hat{Z}}, b):=\ul{\breve{\CM}}_{\ul\BL}^{\hat{Z}}$ to the local $\nabla\scrH$-datum $(\BP,{\hat{Z}}, b)$. 

\end{definition}

The following theorem ensures the representability of the above functor by a formal scheme locally formally of finite type.

\begin{theorem} \label{ThmRRZSp}
In the above situation if $\BP$ is a smooth affine group scheme over $\BD$ with connected reductive generic fiber, the functor 
$\ul{\breve{\CM}}_{\ul\BL}^{\hat{Z}}$ is ind-representable by a formal scheme $\breve{\CM}_{\ul\BL}^{\hat{Z}}$ over $\Spf \BaseFldOfLocSht\dbl\xi\dbr$ which is locally formally of finite type and separated. It is called a \emph{bounded Rapoport-Zink space for local $\BP$-shtukas}. Its underlying reduced subscheme equals the associated \emph{affine Deligne--Lusztig variety}, which is the reduced closed ind-subscheme $X_Z(\ul\BL)\subset\SpaceFl_\BP\whtimes_\BaseOfD\Spec\BaseFldOfLocSht$ whose $K$-valued points (for any field extension $K$ of $\BaseFldOfLocSht$) are given by
\begin{equation}
\label{EqADLV}
X_Z(\ul\BL)(K)\;:=\;X_Z(b)(K)\;:=\;\big\{ g\in \SpaceFl_\BP(K)\colon g^{-1}\,b\,\hat\sigma^\ast(g) \in Z(K)\big\}.
\end{equation}
In particular $X_Z(\ul\BL)$ is a scheme locally of finite type over $\BaseFldOfLocSht$. Its irreducible components are quasi-projective schemes over $k$. Moreover, they are
projective if $\BP$ is parahoric in the sense of Bruhat and Tits (in the sense of \cite[D\'efinition 5.2.6]{B-T}).
\end{theorem}

\begin{proof}

See Theorem~4.18 and Corollary~4.26 of \cite{AH_Local}.

\end{proof}

\begin{remark}\label{RemJActsOnRZ}
The group $\QIsog_{\BaseFldOfLocSht}(\ul\BL)$ of quasi-isogenies of $\ul\BL$ acts on $\breve{\CM}_{\ul\BL}^{\hat{Z}}$ via $j\colon(\ul\CL,\bar\delta)\mapsto(\ul\CL,j\circ\bar\delta)$. When $\ul\BL=\bigl((L^+\BP)_\BaseFldOfLocSht,b\hat{\sigma}^*\bigr)$ is trivialized and decent, $\QIsog_{\BaseFldOfLocSht}(\ul\BL)=J_b(\BaseOfD\dpl z\dpr)$ where $J_b$ is the connected algebraic group over $\BaseOfD\dpl z\dpr$ which is defined by its functor of points that assigns to an $\BaseOfD\dpl z\dpr$-algebra $R$ the group
\begin{equation}\label{EqGroupJ}
J_b(R):=\bigl\{\,j \in \genericG(R\otimes_{\BaseOfD\dpl z\dpr} {\BaseFldOfLocSht\dpl z\dpr})\colon j^{-1}b\hat{\sigma}(j)=b\,\bigr\}\,,
\end{equation}
see \cite[\RemJb]{AH_Local}.
\end{remark}

\section{The local model theorem}\label{SecLMT}

In \cite{AH_LM} the authors proved two versions of local model theorem for moduli stack $\nabla_n^{H,\ul{\hat{Z}}}\scrH^1(C,\FG)^\ul\nu$ of global $\FG$-shtukas corresponding to $\nabla\scrH$-datum $(\FG, \ul{\hat{Z}}, H)$; see remark \ref{RemNablaHData} and \cite[Theorem 4.4.6 and Theorem 3.2.1]{AH_LM}. See also \cite[Theorem 2.20]{Var} for constant case, i.e. when $\FG=G\times_{\BF_q}C$ for a split reductive group $G$ over $\BF_q$. These theorems ensure that the local geometry of $\nabla_n^{H,\ul{\hat{Z}}}\scrH^1(C,\FG)^\ul\nu$ can be completely described by the corresponding boundedness conditions. Below we formulate and prove the local model theorem for Rapoport-Zink spaces for local $\BP$-shtukas.

\begin{theorem}\label{ThmRapoportZinkLocalModel}

Consider the assignment \ref{eqAssignRZToLD}. To a local $\nabla\scrH$-datum $(\BP,{\hat{Z}},b)$ one can assign a roof

\begin{equation}\label{EqnablaHRoof} 
\xygraph{
!{<0cm,0cm>;<1cm,0cm>:<0cm,1cm>::}
!{(0,0) }*+{\wt{\CM}_{\ul\BL}^{\hat{Z}}}="a"
!{(-1.5,-1.5) }*+{\breve{\CM}}="b"
!{(1.5,-1.5) }*+{{\hat{Z}} ,}="c"
"a":_{\pi}"b" "a":^{\pi^{loc}}"c"
}  
\end{equation}
\noindent
where $\breve{\CM}:=\breve{\CM}(\BP,{\hat{Z}},b):=\breve{\CM}_{\ul\BL}^{\hat{Z}}$, that satisfies the following properties

\begin{enumerate}
\item
the morphism $\pi^{loc}$ is formally smooth and
\item
$\wt{\CM}_{\ul\BL}^{\hat{Z}}$ is an $L^+\BP$-torsor under $\pi: \wt{\CM}_{\ul\BL}^{\hat{Z}}\to \breve{\CM}$. It admits a section $s'$ locally for the \'etale topology on $\breve{\CM}$ such that $\pi^{loc}\circ s'$ is formally \'etale.
\end{enumerate}

\end{theorem}

\begin{proof} 

Define $\wt{\CM}_{\ul\BL}^{\hat{Z}}$ to be the space associated to the following functor of points

\begin{eqnarray}\label{EqRecBd}
\wt{\CM}_{\ul\BL}^{\hat{Z}}:&(\Nilp_{R_{\hat{Z}}})^o &\longto \Sets  \hspace{7cm}\vspace{-2mm}\nonumber\\
&\SSS &\longmapsto \big\{ (\ul\CL:=(\CL_+,\tau),\delta, \gamma); \text{consisting of}\\ & & ~~~~~~~~~~~~~~~~~~~~~~~~~~~~-(\ul\CL,\delta) \in \breve{\CM}_{\ul\BL}^{\hat{Z}} \text{and}\nonumber\\
& & ~~~~~~~~~~~~~~~~~~~~~~~~~~~~-\text{ a trivialization}~\gamma:\hat{\sigma}^\ast\CL_+\tilde{\to}L^+\BP
\big\}.\nonumber
\end{eqnarray}
\noindent

Sending the tuple $(\ul\CL:=(\CL_+, \tau_{\CL}),\delta,\gamma)$ to $(\CL_+, \gamma\circ \tau_{\CL}^{-1})$ defines a map ${\wt{\CM}_{\ul\BL}^{\hat{Z}}} \to \wh{\CF\ell}_{G,R_{\hat{Z}}}$. As the local $\BP$-shtuka $\ul\CL$ is bounded by $\hat{Z}$, this morphism factors through ${\hat{Z}}$; see Definition \ref{EqRecBd}. This defines the map $\pi^{loc}:{\wt{\CM}_{\ul\BL}^{\hat{Z}}} \to {\hat{Z}}$.\\
Take a closed immersion $i: S_0\to S$ defined by a nilpotent sheaf of ideals $I$. Since $I$ is nilpotent, there is a morphism $j:S\to S_0$ such that the $q$-Frobenius $\sigma_S$ factors as follows

\[
\xymatrix {
S\ar[r]^j \ar@/^2pc/[rr]^{\sigma_S} & S_0 \ar[r]^i& S.
}
\]
\noindent
Let $(\CL_{0+}, \tau_{\CL_0},\delta_0,\gamma_0)$ be a point in $\wt{\CM}_{\ul\BL}^{\hat{Z}}(S_0)$ and assume that it maps to $(\CL_{0+},g_0)$ under $\pi^{loc}$. Furthermore assume that $(\CL_+,g:\CL\tilde{\to} (LP)_S)$  lifts $(\CL_{0+},g_0)$ over $S$, i.e. $i^\ast \CL=\CL_0$ and $i^\ast g= g_0= \gamma_0\circ \tau_{\CL_0}^{-1}$. \\
\noindent
Consider the following diagram

\[
\xymatrix {
S_0\ar[rr]^{(\CL_{0+},\tau_{\CL_0},\delta_0,\gamma_0)}\ar[d]_i & & \wt{\CM}_{\ul\BL}^{\hat{Z}}\ar[d]^{\pi^{loc}}  \\
S\ar[rr]_{(\CL_+,g)}\ar@{.>}[urr]^{\alpha}& & {\hat{Z}} \;.
}
\]

To prove $a)$ we have to verify that there is a map $\alpha$ that fits in the above  commutative diagram.

\noindent
We construct $\alpha : S\to\wt{\CM}_{\ul\BL}^{\hat{Z}}$ in the following way. First we take a lift $\gamma:\sigma_S^\ast \CL_+\tilde{\to} (L^+\BP)_S$ of $\gamma_0:\sigma_{S_0}^\ast \CL_{0+}\tilde{\to} (L^+\BP)_{S_0} $. To see the existence of such lift one can proceed as in  \cite[Proposition~2.2.c)]{H-V}. Namely, regarding the smoothness of $\BP$, one first observes that if a torsor gets mapped to the trivial torsor under $$\check{H}^1(S_{\text{\'et}},L^+\BP)\to\check{H}^1(S_{0, \text{\'et}},L^+\BP),$$ it must initially be a trivial one. Consider an $L^+\BP$-torsor $\CL_+$ over $S$. It can be represented by trivializing cover $S'\to S$ and an element $h''\in L^+\BP(S'')$, where $S''=S'\times_S S'$. A given trivialization $\gamma_0$ of $\CL_+$ over $S_0$ is given by $g_0' \in L^+\BP (S_0')$ with $p_2^\ast(g_0')p_1^\ast(g_0')^{-1}=h_0''$, where $h_0''$ is the image of $h''$ under $L^+\BP(S'')\to L^+\BP(S_0'')$ and $p_i: S''\to S'$ denotes the projection to the $i$'th factor, $i=1,2$. Take a trivialization $\beta$ of $\CL_+$, given by $f'\in L^+\BP(S')$ with $p_2^\ast(f')p_1^\ast(f')^{-1}=h''$. We modify it in the following way. Let $f_0'$ be the restriction of $f'$ to $S_0$. We have $p_2^\ast(f_0'^{-1}g_0')=p_1^\ast(f_0'^{-1}g_0')$ and therefore $f_0'^{-1}g_0'$ induces $t_0 \in L^+\BP(S_0)$. By smoothness of $\BP$ this section lifts to $t\in L^+\BP(S)$. The element $g':= t \cdot f'$ lifts $g_0'$ and satisfies $p_2^\ast(g')p_1^\ast(g')^{-1}=h''$, thus induces the desired trivialization $\gamma$. This ensures the existence of the lift $\gamma$. \\
The morphism $\alpha$ is given by the following tuple
$$
(\CL_+,\tau_\CL,\delta,\gamma):=(\CL_+,g^{-1}\circ \gamma,\tau_{\BL}\circ j^\ast \delta_0\circ\gamma^{-1}\circ g, \gamma).
$$

\noindent
Notice that

\begin{eqnarray*}
i^\ast(\CL_+,\tau_\CL,\delta,\gamma)&=&i^\ast (\CL_+,g^{-1}\circ \gamma,\tau_{\BL}\circ j^\ast \delta_0\circ\gamma^{-1}\circ g, \gamma)\\
&=&(\CL_{0+},i^\ast g^{-1}\circ\gamma_0, \tau_{\BL}\circ\sigma^\ast\delta_0\circ\tau_{\CL_0}^{-1},\gamma_0)\\
&=&(\CL_{0+},\tau_{\CL_0},\delta_0,\gamma_0)
\end{eqnarray*}

\noindent
and that $\pi^{loc}(\CL_+,\tau_\CL,\delta,\gamma)=(\CL,g)$.

\noindent
Now we prove part b). We take an \'etale covering $\CM'\to \breve{\CM}_{\ul\BL}^{\hat{Z}}$ such that the universal $L^+\BP$-torsor  $\CL_+^{univ}$ admits a trivialization $\gamma': \CL_{+,\CM'}^{univ}\tilde{\to}(L^+\BP)_{\CM'}$. For existence of such trivializing covering see \cite[Proposition 2.4]{AH_Local}. This yields the section

\begin{equation}\label{EqM'}
~~~\xymatrix {
& & \wt{\CM}_{\ul\BL}^{\hat{Z}} \ar[dl]_{\pi}\ar[dr]^{\pi^{loc}} &  \\
\CM'\ar[r]\ar@/^2pc/[urr]^{s'}&\breve{\CM}_{\ul\BL}^{\hat{Z}}& &\hat{Z}  \;.
}
\end{equation}

\noindent
corresponding to the tuple $(\CL_+^{univ},\delta,\sigma^\ast \gamma')$.
\noindent
Consider the following diagram

\[
\xymatrix {
S_0\ar[rr]^{(\CL_{0+},\tau_{\CL_0},\delta_0,\gamma_0')}\ar[d]_i & & \CM'\ar[d]^{\pi^{loc}\circ s'}  \\
S\ar[rr]_{(\CL_+,g)}\ar@{.>}[urr]^{\alpha'}& &\hat{Z} \;.
}
\]
\noindent
We want to find $(\CL_+,\tau_{\CL},\delta,\gamma')$ with $g=\sigma^\ast \gamma'\tau_\CL^{-1}.$
First we construct
$$
(\CL_+,\tau_\CL,\delta,\gamma)\in \wt{\CM}_{\ul\BL}^{\hat{Z}}(S).
$$
Since $\sigma^\ast\gamma'=j^\ast i^\ast\gamma'=j^\ast\gamma_0'$, we take $\gamma:=j^\ast\gamma_0'$. This gives the morphism $\delta$ according to the following commutative diagram

\[
\xymatrix {
\BL & & \CL\ar[ll] \ar[rd]^{g}  \\
\sigma^\ast\BL \ar[u]_{\tau_\BL}& &\sigma^\ast\CL\ar[u]\ar[ll]^{j^\ast \delta_0} \ar[r]_{j^\ast \gamma_0'}^{\sim} & (LP)_S \;.
}
\]

\noindent
and furthermore determines $\tau_\CL$, we set 
$$
y:=\left(\CL_+,g^{-1}j^\ast\gamma_0',\tau_{\BL}\circ j^\ast \delta_0 \circ j^\ast\gamma_0'^{-1}\circ g,j^\ast \gamma_0'\right)\in 
\wt{\CM}_{\ul\BL}^{\hat{Z}}(S)
$$
with $\pi^{loc}(y)=(\CL_+, g)\in \hat{Z}(S)$. The section $s'$ sends $(\CL_{0+},\tau_{\CL_0},\delta_0,\gamma_0')$ to 
$$
(\CL_{0+},\tau_{\CL_0},\delta_0,\gamma_0=i^\ast j^\ast \gamma_0'=\sigma_{S_0}^\ast \gamma_0')=i^\ast y\in\wt{\CM}_{\ul\BL}^{\hat{Z}}(S_0).
$$
\noindent
Consider the point $$\pi(y)=(\CL_+,\tau_\CL,\delta)\in \breve{\CM}_{\ul\BL}^{\hat{Z}}(S)$$ with $i^\ast \pi(y)= (\CL_{0+},\tau_{\CL_0},\delta_0)$. Then, since $\CM'\to \breve{\CM}_{\ul\BL}^{\hat{Z}}$ is \'etale, there is a unique $\gamma':\CL\tilde{\to}(L^+\BP)_S$ with $i^\ast\gamma'=\gamma_0'$. Note that $$\gamma:=\sigma^\ast \gamma'=j^\ast i^\ast \gamma'=j^\ast \gamma'.$$
\noindent
This ensures the existence of $\alpha'$ which is given by $(\CL_+,\tau_\CL,\delta,\gamma')$.
To see the uniqueness let $(\CL_+,\tau_\CL,\delta,\gamma')\in\CM'(S)$ with $i^\ast(\CL_+,\tau_\CL,\delta,\gamma')=(\CL_0,\tau_{\CL_{0+}},\delta_0,\gamma'_0)$ and

$$
\pi^{loc}(\CL_+,\tau_\CL,\delta,\gamma'):=(\CL_+,\sigma^\ast \gamma' \tau_{\CL}^{-1})=(\CL,g).
$$

\noindent
Therefore $\CL_+$, $\tau_{\CL}=g^{-1} \sigma^\ast \gamma' = g^{-1} j^\ast \gamma_0'$ and $\delta=\tau_\BL \circ j^\ast\delta_0\circ j^\ast\gamma_0'^{-1}\circ g$ are uniquely determined and provide a point in $\breve{\CM}_{\ul\BL}^{\hat{Z}}(S)$ and then $\gamma'$ is also uniquely determined.
\end{proof}


\section{Applications}\label{Applications}
Below in subsections \ref{SubSectLocal properties of R-Z spaces} and \ref{SubSectRelToADLV}, we discuss some immediate consequences of the local model theorem \ref{ThmRapoportZinkLocalModel}. Then in subsection \ref{Subsect Semi-simple trace and Frob}, using the local model theory, the theory of formal nearby cycles \cite{Mieda} together with the uniformization theory of the stack of global $\FG$-shtkas \cite{AH_Global}, we study the nearby cycles cohomology of these moduli spaces.

\subsection{Local properties of R-Z spaces}\label{SubSectLocal properties of R-Z spaces}

As we will see below, the local model theorem has an immediate corollary, that describes the local geometry of Rapoport-Zink spaces, compare \cite[Proposition 4.5.2]{AH_LM}.  Moreover, in a similar way to the local model theory for moduli stacks of global $\FG$-shtukas, see \cite[Proposition 4.5.3]{AH_LM} (and also \cite[Theorem 3.21]{AH_MR}), it also has some global consequences, which we partly explain below.

\begin{definition}\label{DefPsatisfiesSerreCondition}

Consider the Serre conditions $S_i$ and $R_i$ in the sense of \cite[IV, \S 5.7 and 5.8]{EGA}. 
We say that a group $\BP$ satisfies $SS_i$ (resp. $SR_i$) if all singularities occurring in the Schubert varieties, i.e. closures of the orbits under the $L^+\BP$-action on $\CF\ell_\BP$, satisfy  $S_i$ (resp. $R_i$). Similarly we say that $\BP$ is S-CM (resp. S-N) if all singularities occurring in the orbit closures of the orbits under the $L^+\BP$-action are Cohen-Macaulay (resp. normal).
\end{definition}

\begin{remark}\label{RemparahoricPis}
A parahoric group $\BP$ with tame generic fiber $P$ satisfies $SS_i$ for all $i$, as well as $SR_0$ and $SR_1$, see \cite[Theorem~8.4]{PR2}, according to Serre's criterion for normality.
\end{remark}

\begin{corollary}\label{PropLocModel}
We have the following statements:\\
\begin{enumerate}

\item[a)]
The Rapoport-Zink space $\breve{\CM}_{\ul\BL}^{\hat{Z}}$ satisfies $S_i$ (resp. $R_i$) if $\hat{Z}$ satisfies $S_i$ (resp. $R_i$). In particular $\breve{\CM}_{\ul\BL}^{\hat{Z}}$ satisfies $S_i$ if $\BP$ satisfies $SS_i$ and $\zeta$ is not a zero divisor in $\CO_{\hat{Z}}$. 

\item[b)]
The Rapoport-Zink space $\breve{\CM}_{\ul\BL}^{\hat{Z}}$ is flat over its reflex ring $R_{\hat{Z}}$ iff $\hat{Z}$ is flat. The latter is the case when $\BP$ is S-CM and $\zeta$ is not a zero divisor in $\CO_{\hat{Z}}$.

\end{enumerate}
\end{corollary}

\begin{proof}
The first statement of part a) follows from Theorem \ref{ThmRapoportZinkLocalModel} and the fact that satisfying $R_i$ (resp. $S_i$) is an \'etale local property. For the second statement, since the element $\zeta$ is regular and by definition $Z$ satisfies $S_i$, we argue that $\hat{Z}$ satisfies $S_i$. \\     
The first statement of part b) is clear according to Theorem \ref{ThmRapoportZinkLocalModel} and that the henselization morphism $R\to R^h$ is faithfully flat. Note that $\breve{\CM}_{\ul\BL}^{\hat{Z}}$ is locally formally of finite type, see Theorem \ref{ThmRRZSp}. For the second statement, first observe that $\hat{Z}$ is Cohen-Macaulay according to \cite[Theorem~8.4]{PR2} and \cite[Theorem 17.3]{Matsumura}. Now the statement follows from \cite[IV, Proposition 6.1.5]{EGA}.

\end{proof}

\subsection{Relation to affine Deligne-Lusztig varieties}\label{SubSectRelToADLV} 

Fix a local $\nabla\scrH$-datum $(\BP,\hat{Z},b)$ and consider the corresponding Rapoport-Zink space $\breve{\CM}:=\breve{\CM}(\BP,\hat{Z},b)$, see \ref{eqAssignRZToLD}.

The following corollary can be viewed as a local version of \cite[Proposition 2.8]{VLaff}.

\begin{corollary}
The induced morphism $\breve{\CM}\to [L^+\BP\backslash \hat{Z}]$ of formal algebraic stacks, is formally smooth.
\end{corollary}

\begin{proof}
This immediately follows from local model theorem \ref{ThmRapoportZinkLocalModel}. For discussions about formal stacks see \cite[Appendix A]{Har1} or \cite[Section~2.1]{EsmailDissertation}.
\end{proof}

Assume that $\BP$ is a parahoric group scheme. This in particular implies that $\CF\ell_\BP$ is ind-projective. The above morphism induces a morphism $\breve{\CM}_s\to [L^+\BP\backslash Z]$ on the special fibers. As a set, $[L^+\BP\backslash Z]$ is given by a set of representatives $\{g_\omega\}_\omega$ corresponding to the orbits of the $L^+\BP$-action on $Z$, see Remark \ref{RemAffSchVarandIwahori_Weylgp}. This is indexed by a finite subset of the affine Weyl group $\wt W$ associated with $\BP$. We denote this subset by $Adm(Z)$. 
The special fiber $\breve{\CM}_s$ can be written as the  union of the following affine Deligne-Lusztig varieties
$$
X_Z(b)^\omega := \{g \in \CF\ell_\BP(\ol \BF) ; g^{
-1}b\sigma^\ast(g) \in S_\omega\}
$$
where $\omega$ lies in $ Adm(Z)$ and $S_\omega$ denotes the preimage of $\omega$ under the map $Z\to [L^+\BP\backslash Z]$.

\subsection{Formal nearby cycles}\label{Subsect Semi-simple trace and Frob}

Below we first discuss the nearby cycles for the case of global moduli stacks for $\FG$-shtukas, and then we disccuss the formal nearby cycles for the Rapoport-Zink spaces for local $\BP$-shtukas.  

\subsubsection{The case of the moduli stacks for global $\FG$-shtukas}

Recall from \cite{AH_LM} that the moduli stack $\nabla_n^\CZ\scrH_D^1(C,\FG)$ of global $\FG$-shtukas with $D$-level structure, for a divisor $D\subseteq C$, bounded by a global boundedness condition $\CZ$, in the sense of \cite[Definition 3.1.3]{AH_LM}, is an algebraic stack whose $T$-points, for $\BF_q$-scheme $T$, is given by the groupoid whose objects are tuples $(\CG, \psi, \ul s, \tau )$ consisting of a $\FG$-bundle $\CG$ over $C_T$, a trivialization $\psi: \CG \times_{C_T} D_T \isoto \FG \times_C D_T$, an n-tuple $\ul s\in C^n(T)$ of (characteristic) sections, and finally an isomorphism $\tau : \sigma^\ast \CG|_{C_T\setminus \Gamma_{\ul s}} \to \CG|_{C_T\setminus \Gamma_\ul s}$ with $\psi\circ\tau = \sigma^\ast(\psi)$, that is bounded by $\CZ$. Here $\Gamma_\ul s$ denotes the union $\cup_i\Gamma_{s_i}$ of the graphs $\Gamma_{s_i}$.
For the sake of simplicity let us assume that the reflex curve $C_\CZ$, see \cite[Definition 3.1.3]{AH_LM}, equals $C$ itself. There is a canonical map $\nabla_n^\CZ\scrH_D^1(C,\FG)\to C^n$ given by sending $(\CG, \psi, \ul s, \tau )$ to $\ul s$. Note that this stack is Deligne-Mumford \cite[Theorem 3.1.6]{AH_LM} and the forgetful morphism $\nabla_n^\CZ\scrH_D^1(C,\FG)\to \nabla_n^\CZ\scrH^1(C,\FG)$ is finite \'etale \cite[Theorem 6.7]{AH_Global}. Moreover, one can observe that for $D$ enough big, the stack $\nabla_n^\CZ\scrH_D^1(C,\FG)$ can be covered by quasi-projective open subschemes $\nabla_n^\CZ\scrH_D^1(C,\FG)_\alpha$; see \cite[Remark 2.9 and Theorem 3.15]{AH_Global}.

We proved in \cite[Theorem 3.2.1]{AH_LM} that $\CZ$ is a local model for the moduli stack $\nabla_n^\CZ\scrH_D^1(C,\FG)$. Fix a closed immersion $\delta: C\to C^n$, and set $\CX_\delta:=\nabla_n^\CZ\scrH_D^1(C,\FG)_\alpha\times_{C^n,\delta}C$ and $\CZ_\delta=\CZ\times_{C^n,\delta} C$. Furthermore, set $\CX_{\delta}^\nu:= \CX_\delta\times_C S$ and $ \CZ_{\delta}^\nu:=\CZ_\delta\times_C S$, where $S=\Spec A_{\nu}$ for a place $\nu$ on $C$. Assume further that $\CZ_{\delta}^\nu$ (and hence $\CX_{\delta}^\nu$), is flat over $S$. Let $s$ (resp. $\eta$) denote the special (resp. generic) point of $S$. Let $\kappa(s)$ (resp.   $\kappa(\eta)$), denote the corresponding residue fields and let $\ol S$ denote the formal spectrum of the integral closure of $A_{\nu}$ inside the separable closure of $\kappa(\eta)$, with $\ol s$ (resp. $\ol \eta$) the corresponding special (resp. generic) point. 
For $\CF$ in the bounded derived category $D_c^b(\CX_\eta^\nu, \ol\BQ_\ell)$, consider the usual nearby cycles sheaf $R\psi_{\CX_\delta}\CF=\ol i^\ast R \ol j_\ast \CF_\ol\eta$, where $\ol i: \CX_{\delta\ol s}^\nu\to  \CX_{\delta\ol S}^\nu$
and $j : \CX_{\delta\ol\eta}^\nu \to \CX_{\delta\ol S}^\nu$
are the closed and open immersions of the geometric special
and generic fibers of $\CX_\delta^\nu$ over $S$, and $\CF_\ol\eta$ is the pull-back of $\CF$ to $\CX_\ol\eta^\nu$. One similarly defines the nearby cycles sheaf $R\psi_{\CZ_\delta^\nu}\CF$. According to \cite[Theorem 3.2.1]{AH_LM} for a point $x$ in $\CX_\delta(\kappa_r)$, we have a roof of \'etale morphisms

\begin{equation}
\xygraph{
!{<0cm,0cm>;<1cm,0cm>:<0cm,1cm>::}
!{(0,0) }*+{\CU_x}="a"
!{(-1.5,-1.5) }*+{\CX_\delta^\nu}="b"
!{(1.5,-1.5) }*+{\CZ_\delta^\nu,}="c"
"a":_{\pi}"b" "a":^{\pi^{loc}}"c"
}  
\end{equation}
\noindent
which shows that there is an isomorphism $(R\psi_{\CX_\delta^\nu}\ol\BQ_\ell)_x\isoto (R\psi_{\CZ_\delta^\nu}\ol\BQ_\ell)_y$ of the stalks of the nearby cycles sheaves, for $y=\pi^{loc}(\wt x)$ with $\pi(\wt x)=x$. This in particular implies the following equality of semi-simple traces on the stalks of these sheaves. Let us summerize the above descussion in the following proposition.\\

\begin{proposition}
    Keep the above notation. Assume that $\CZ_\delta^\nu$ is flat over $S$. Then for a point $x$ in $\CX_\delta^\nu$ there is a point $y$ in $\CZ_\delta^\nu$ and an isomorphism $(R\psi_{\CX_\delta^\nu}\ol\BQ_\ell)_x\isoto (R\psi_{\CZ_\delta^\nu}\ol\BQ_\ell)_y$ of the stalks of the nearby cycles sheaves. In particular, for a field extension $\kappa_r/\kappa(\nu)$ of degree $r$, we have the following equality 
$$
tr^{ss}(Frob_r; (R\psi_{\CX_\delta}^\nu\ol{\BQ}_\ell)_x)=tr^{ss}(Frob_r; (R\psi_{\CZ_\delta}^\nu\ol{\BQ}_\ell)_y).
$$
of semi-simple traces on the stalks of these sheaves.
\end{proposition}

\begin{remark}
For the discussion related to the flatness of the moduli stacks of bounded global $\FG$-shtukas see \cite[Proposition 4.5.3]{AH_LM}.
\end{remark}

\bigskip
\subsubsection{Formal nearby cycles for Rapoport-Zink spaces}
In the remaining part of this paper we discuss the analogous result for Rapoport-Zink spaces for local $\BP$-shtukas, using the local model theorem \ref{ThmRapoportZinkLocalModel}. The crucial difference is that we need to work with \emph{formal} nearby cycles (with compact support), rather than the usual nearby cycles and this makes the situation slightly more complicated. We discuss this below.\\

Throughout this subsection we assume that the group $\BP$ is a parahoric group scheme over $\BD$, in the sense of Bruhat and Tits \cite[D\'efinition 5.2.6]{B-T}. Moreover, in addition to the conditions stated in Definition \ref{DefLBC}\ref{DefBDLocal_C}, we assume that the boundedness condition $\hat{Z}$ also satisfies the following axioms 
\begin{enumerate}
    \item[i)] $\hat{Z}_R$ is a $\zeta$-adic formal scheme over $\Spf R$.
\item[ii)] There is a faithful representation $\rho: \BP\to \SL_r$ over $\BF_q\dbl z\dbr$ and a positive integer $n$ such that all the induced morphism $\rho_\ast : \hat{Z}_R\to \wh{\CF\ell}_{SL_r}$ factor through $\wh{\CF\ell}_{SL_r}^{(n)}$. Here $\wh{\CF\ell}_{SL_r}^{(n)}$ is the closed ind-subscheme of $\wh{\CF\ell}_{SL_r}$
given by
$$
\CD
\wh{\CF\ell}_{SL_r}^{(n)}(S):= \{(\CL_+,\delta: \CL\to (L\SL_r)_S)\in \wh{\CF\ell}_{SL_r}(S); \text{such that}\\
~~~~~~~\wedge_{\CO_S\dbl z \dbr}^j M(\delta)(M(\CL_+))\subset (z-\zeta)^{n(j^2-jr)}\wedge_{\CO_S\dbl z \dbr}^jM((L^+\SL_r)_S)\},
\endCD
$$
see Proposition \ref{PropFlRepUnBoundedRZ}. Here $M(\CL_+)$ denote the pair $(M, \alpha)$ corresponding to $L^+\SL_r$-torsor $\CL_+$, that consists of a finite locally free
$\CO_S\dbl z\dbr$-module of rank $r$ on $S$ and a trivialization $\alpha: \wedge_{\CO_S\dbl z \dbr}^r\tilde{\to} \CO_S\dbl z\dbr$.
Note that $M(-)$ is an equivalence of categories.   

\item[iii)]  Let $(\hat{Z}_R)^{an}$
be the strictly $R\dbl 1/\zeta\dbr$-analytic space associated with $\hat{Z}_R$. It can be shown that there is a closed
subscheme $\hat{Z}_E$ of the affine Grassmannian $Gr_{\BP}^{\textbf{B}_\text{dR}}
 \times_{\BF_q\dpl \zeta\dpr} \Spec E_{\hat{Z}}$, with $E_{\hat{Z}}=Frac(R_{\hat{Z}})$, such that $(\hat{Z}_R)^{an}$ arises by
base change to $R\dbl
\zeta\dbr$ from the strictly $E_{\hat{Z}}$-analytic space $(\hat{Z}_E)^{an}$ associated with $\hat{Z}$. One requires that $\hat{Z}_E$, and hence also all the $(\hat{Z}_R)^{an}$ are invariant under the left multiplication of $\BP(_\bullet \dbl z-\zeta\dbr)$ in $Gr_{\BP}^{\textbf{B}_{dR}}$. 
 Note that in the function fields set up the affine Grassmannian $Gr_{\BP}^{\textbf{B}_{dR}}$ is the ind-scheme corresponding to the sheaf of sets for the fpqc topology on  $\BF_q\dpl\zeta\dpr$ associated with the presheaf $X\mapsto P(\CO_X\dpl z-\zeta\dpr)/\BP(\CO_X\dbl z-\zeta\dbr)$. See \cite[Chapter 4]{HV21}.
\end{enumerate}

\noindent
suggested by Hartl and Viehmann in \cite[Definition 2.2]{HV21}, and in addition, we assume that

\begin{enumerate}
    \item[iv)] the formal schemes $\hat{Z}_R$ are flat over $\Spf R$.
    
\end{enumerate}

\begin{remark}

Note that the above additional conditions i)-iii) are in fact designed to fulfill the following desire. Namely, it can be seen easily that these conditions imply that $\hat{Z}_R$ are $\zeta$-adic formal schemes, projective over $\Spf R$, and moreover, the associated $R[1\slash \zeta]$-analytic spaces $(\hat{Z}_R)^{an}$
arise from an strictly $E_{\hat{Z}}$-analytic space  $\hat{Z}^{an}:=(\hat{Z}_E)^{an}$ 
associated with a projective scheme $\hat{Z}_E$ over $\Spec E_{\hat{Z}}$, which is a closed subscheme of the affine Grassmannian $Gr_{\BP}^{\textbf{B}_\text{dR}}\times_{\BF_q\dpl\zeta\dpr} \Spec E_{\hat{Z}}$, see \cite[Proposition 2.6]{HV21} for the details. Also the last condition is essential to the theory of nearby and vanishing cycles. Note however that for a tame group $G$, when the boundedness condition $\hat{Z}$ comes from a global boundedness condition, determined by a cocharacter $\mu$ of $G$, then this condition automatically holds, e.g. see \cite{Zhu}. Here let us briefly recall the context, namely, the cocharacter $\mu$ defines a global boundedness condition, see \cite[Definition 3.1.3]{AH_LM}, that corresponds to the Schubert variety $\CZ:=\CZ_\mu$ inside the global affine Grassmannian $GR_1(C,\FG)$, which is by definition the closure of the Schubert variety $\CS(\mu)$, lying in the generic fiber $GR_1(C,\FG)_\eta$. For a place $\nu$ on $C$, set $\BP:=\BP_\nu:=\FG\wh{\times}_C A_\nu$, and then $\hat{Z}=\hat{Z}_\mu$ is defined to be the local boundedness condition associated with $\CZ$, see \cite[Prop. 4.3.3]{AH_LM}, and Remark \ref{RemAffSchVarandIwahori_Weylgp}. Note in addition that this coincides the local model $M_\mu$ in the sense of \cite[Definition 2.5]{Richarz16}, and note further that since $\BP$ is parahoric, the formal scheme $\hat{Z}$ can be thought as a proper scheme over $\Spec A_\nu$, and the formal nearby cycles sheaf $R\Psi_{\hat{Z}}\BQ_\ell$ coincide the usual nearby cycles sheaf $R\psi_{\hat {Z}}\BQ_\ell$ for schemes.   
\end{remark}


Below we will prefer to use a variant of the Berkovich's formal nearby cycles sheaf for a formal scheme $\FX$ over a complete discrete valuation ring $R$, which has been introduced and studied by Y. Mieda in \cite{Mieda}. As a technical feature of this theory, one may use it in order to study the \emph{compactly supported} nearby cycles $R\Psi_{\FX,c}\Lambda$ of Rapoport-Zink spaces. Note that according to \cite[Theorem 1.1. (iv)]{Mieda}, for a locally algebrizable and pseudo-compactifiable $\FX$, the compactly supported cohomology $H_c^q(\FX_{\ol\eta},\Lambda)$ of the geometric generic fiber $\FX_{\ol\eta}$ coincides the compactly supported cohomology \forget{$H_c^q(\FX_{red},R\Psi_{\FX,c}\Lambda)$}of $\FX_{red}$, with coefficients in $R\Psi_{\FX,c}\Lambda$, where $\Lambda=\BZ/\ell^n\BZ$ with a prime $\ell$ invertible in $R$. Note however that this is not the case with the Berkovich's nearby cycles in general, e.g. see \cite[Remark 4.30]{Mieda}. In this section, using the uniformization theory of global $\FG$-shtukas \cite{AH_Global}, we also discuss a compatibility result with the Berkovich's formal nearby cycles in certain cases. 
Before discussing the Mieda's variant of Berkovich's formal nearby cycles, let us recall the following

\begin{remark}\label{RemBerNVCandRZSpaces}
Let $\CS:=\Spf R$ be a formal spectrum of a complete discrete valuation ring $R$ and let $s$ (resp. $\eta$) denote the special (resp. generic) point of $\CS$. Let $\kappa(s)$ (resp.   $\kappa(\eta)$), denote the corresponding residue fields. Let $\ol \CS$ denote the formal spectrum of the integral closure of $R$ inside the separable closure of $\kappa(\eta)$, with $\ol s$ (resp. $\ol \eta$) the corresponding special (resp. generic) point. 
For a formal scheme $\FX$, locally of finite presentation over $R$, let $\FX_s$ and $\FX_\eta$ (resp $\FX_\ol s$ and $\FX_\ol \eta$) denote the associated closed and generic fiber of $\FX$ (resp. $\ol \FX:=\FX\times_\CS \ol \CS$). According to \cite{BerkI} one may define the following functor
$$
\CD
R\Psi_\FX^{\textbf{Ber}}: D_b^c(\FX_\ol\eta ,\ol\BQ_\ell)\to D_b^c(\FX_\ol s \times \eta, \ol\BQ_\ell).\\
\endCD
$$
Moreover, there is a spectral sequence 
\begin{equation*}\label{Eq_VanishingCyclesSpectralSeq}
E_2^{p,q}:=\Koh^p(\FX_\ol s, R^q\Psi_\FX^\textbf{Ber}\ol\BQ_\ell)\Rightarrow H^{p+q}(\FX_\ol\eta,\ol\BQ_\ell),
\end{equation*}
which is equivariant under the action of the Galois group $\Gamma = \Gal(\kappa(\ol\eta)/\kappa(\eta))$. Note that the induced filtration $\CW$ on $\CV=\Koh^\ast(\FX_\ol\eta,\BQ_\ell)$ is admissible (in the sense that it is stable under the Weil group action and that the inertia group $I:=\ker\left(\Gamma \to \Gal(\kappa(\ol s)/\kappa(s))\right)$ operates on $gr_{\bullet}^\CW(\CV)$ through a finite quotient). \forget{Recall that for an arbitrary $\BQ_\ell$-representation of $W$ with admissible increasing filtration $\CW$.  The semisimple trace of the j-Frobenius on $\CV$ is defined as follows
$$
tr^{ss}(\sigma^j|\CV):=\sum_k \Tr(\sigma^j|(gr_k^\CW(\CV))^I) 
$$ 
\noindent
 allows to define the semi-simple trace of Frobenius $$tr^{ss}(Fr_q; (R\Psi_\FX^\textbf{Ber}\ol\BQ_\ell)_x)$$ on the stalk of the sheaf of formal nearby cycles $R\Psi_\FX^{\textbf{Ber}}\ol\BQ_\ell$ at $x$. Note however that we unfortunately can not use this method to study Rapoport-Zink spaces for local $\BP$-shtukas, because the cohomology is infinite-dimensional as the space is not quasi-compact, moreover, to implement the Berkovich's formal nearby cycles with compact support, one needs to assume that $\FX$ is formally of finite presentation (i.e. quasi-compact), while, apart from Lubin-Tate case, even the connected components of $\breve{\CM}:=\breve{\CM}_{\ul\BL}^{\hat{Z}}$ may fail to be quasi-compact.}
\end{remark}

 Let us now recall the definition and some basic properties of the formal nearby cycles functor according to \cite{Mieda}.  \\

\point{}\label{MiedaFNC} Let $\FX$ be a locally noetherian formal scheme over $\CS:=\Spf R$, and let $\CI_\FX$ be the largest ideal of definition of $\FX$. Let $\FX_{red}$ denote the associated locally noetherian reduced scheme $(\FX, \CO_\FX/\CI_\FX)$. 
 For an open subscheme $U$ of $\FX_{red}$, we denote by $\FX_{\slash U}$ the open formal subscheme of $\FX$ whose underlying space is $U$. If $U$ is affine then $\FX_{\slash U}$ is also affine and we set $\wh{X}_{\slash U}=\Spec A_U$, where $A_U=\Gamma(\FX_{\slash U},\CO_{\FX})$, see \cite[Lemma 2.5]{Mieda}.

Assume that $\FX$ is quasi-excellent, i.e. for every affine open subscheme $U$ of $\FX_{red}$, the ring $A_U$ is quasi-excellent in the sense of \cite[ Expos\'e I, D\'efinition 2.10]{ILO}, and $\FX_{red}$ is separated. Fix a prime number $\ell$ which is invertible in R, and set $\Lambda:=\BZ\slash\ell^n\BZ$ with $n\geq 0$. Let $\ul\CZ=(\CZ_1,\CZ_2)$ be a pair of closed formal subschemes of $\FX_s$, with $\CZ_1\subseteq \CZ_2$. According to \cite{Mieda}, to such data one assigns the corresponding sheaf of formal nearby cycles  $R\Psi_{\FX,\ul\CZ}\Lambda$ in $D^+(\FX_{red}, \Lambda)$. For this, one first considers an affine open covering $U$ of $\FX_{red}$ and the induced hypercovering $a: U_\bullet \to \FX_{red}$. Then, one needs to observe that $a$ induces a morphism of universally cohomological descent, see \cite[Expos\'e \text{$\text{V}^{\text{bis}}$}, Proposition 3.3.1 (a)]{SGA4}). We let $\wh X_{\slash U_\bullet}$ denote the associated simplicial $S$-scheme, where $S:=\Spec R$. Here we briefly recall the construction of $R\Psi_{\FX,\ul\CZ}\Lambda$.

\begin{definition-remark}\label{Def-RemMFNC} 
\begin{enumerate}
    \item 
Let $\CF = (\CF
^m)_{m\geq 0}$ be a $\Lambda$-sheaf on $U_\bullet$, where $a: U_\bullet\to X$ is a hypercovering of a scheme $X$, corresponding to an open cover $U:=\{U_i\}_{i\in I}$. We say that $\CF$ is \emph{cartesian} if
for every structure morphism $\phi: U_m \to U_n$ of $U_\bullet$, $\phi^\ast\CF^n\to\CF^m$ is an isomorphism. Denote by $D_{cart}^+(U_\bullet, \Lambda)$ the full subcategory of $D^+(U_\bullet, \Lambda)$ consisting of lower bounded complexes whose cohomology are all cartesian. Note that a $\Lambda$-sheaf $\CF=(\CF^m)_{m\geq 0}$ on $U_\bullet$ is cartesian if and only if it arises from a sheaf $\CG$ on $X$ via $a^\ast$, i.e. $\CF\cong a^\ast\CG$. Moreover the functor $Ra_\ast$ gives a quasi-inverse of $a^\ast: D^+(X, \Lambda)\to D_{cart}^+(U_\bullet, \Lambda)$; see \cite[Proposition 2.3]{Mieda}.
\item 
Let $\ul\CZ=(\CZ_1,\CZ_2)$ be a pair of closed formal subschemes of $\FX_s$ as above. Define
$$
R\Psi_{\FX,\ul\CZ,U}\Lambda:= Ra_\ast i^\ast Rj_\ast Rj^! R\psi_{\wh X\slash U_\bullet}\Lambda,
$$
where $i$ (resp. $j$) denote the natural immersion $U_\bullet \to (\wh X_{\slash U_\bullet})_s$ (resp. $j: \wh Z_{\slash U_\bullet}\to (\wh X_{\slash U_\bullet})_s$), and the functor $R\psi_{\wh X\slash U_\bullet}\Lambda$ is defined as the derived push-forward $$ D^+((\wh X_{\slash U_\bullet})_\ol\eta,\Lambda)\to D^+(X_{\slash U_\bullet},\Lambda)$$ followed by pull-back functor $ D^+(\wh X_{\slash U_\bullet},\Lambda)\to D^+((X_{\slash U_\bullet})_s,\Lambda)$. Note that for a $\Lambda$-sheaf $\CF=(\CF^m)_{m\geq0}$ on $(\wh X_{\slash U_\bullet})_\ol\eta$, the restriction of $R\psi_{\wh X_{\slash U_\bullet}}\CF$ to $\wh X_{\slash U_m}$ is the usual nearby cycles complex $R\psi_{\wh X_{\slash U_m}}\CF^m$.
\item 
Set $R\Psi_{\FX,\ul\CZ}\Lambda:= \dirlim[U]R\Psi_{\FX,\ul\CZ,U}\Lambda$, where the inductive limit is taken over the small category whose objects are affine open coverings of $\FX_{red}$.\forget{ and whose morphisms $\Hom(U,V)$ consists of one element for every object $U$ and $V$.} 
\item
One can observe that for a locally noetherian formal scheme $\FX$ with separated $\FX_{red}$, which is locally algebrizable, the formal nearby cycles $R\Psi_{\FX,\ul\CZ} \Lambda$ is a constructible  complex, see \cite[Proposition 3.21]{Mieda}.

\item 
For $\ul \CZ := (\varnothing , \FX_s)$ (resp. $\ul \CZ := (\varnothing, \FX_{red})$), define $ 
R\Psi_{\FX}\Lambda:=R\Psi_{\FX,\ul\CZ}\Lambda$ (resp. $R\Psi_{\FX,c}\Lambda:=R\Psi_{\FX,\ul\CZ}\Lambda)$.

\end{enumerate}

\end{definition-remark}

\begin{proposition}\label{PropnearbycyclesandhatZ}
Assume that $\BP$ is parahoric and let $\hat{Z}:= [(R,\hat{Z}_R)]$ be a bound, then we have

$$
H_c^q((\hat{Z}_R)_\ol\eta,\Lambda)= H^q(Z_\ol s,R\psi_{\hat{Z}_R}\Lambda).
$$

\end{proposition}

\begin{proof}

Note that by the assumption that $\BP$ is parahoric, $\wh{\CF\ell}_{\BP,R}$ is ind-projective \cite[Theorem A]{Richarz16}, and thus
 $\hat{Z}_R$ is algebrizable. Therefore one can see that the sheaf of formal nearby cycles $R\Psi_{\hat{Z}_R}\Lambda$ coincide the corresponding scheme theoretic nearby cycles sheaf $R\psi_{\hat{Z}_R}\Lambda$; e.g. see \cite[Corollary 5.3]{BerkI}. Now the statement follows from \cite[Theorem 1.1 iv) and Proposition 2.12]{Mieda}.

\end{proof}

\point{}\label{Point1} Let us fix an integer $n$ and consider complete discrete valuation rings $\BF_i\dbl z_i\dbr$ for $i=1,\ldots,n$ with finite residue fields $\BF_i$, and fraction fields $Q_i=\BF_i\dpl z_i\dpr$. Let $\BP_{i}$ be a smooth affine group scheme over $\Spec\BF_i\dbl z_i\dbr$ with connected reductive generic fiber $\genericG_i:=\BP_i\times_{\BF_i\dbl z_i\dbr}\Spec\BF_i\dpl z_i\dpr$, and let $\hat{Z}_{i}=[\hat{Z}_{i,R'_i}]$ with $\hat{Z}_{i,R'_i}\subset\wh{\SpaceFl}_{\BP_{i},R'_i}:=\SpaceFl_{\BP_i}\whtimes_{\BF_i}\Spf R'_i$ be a bound in the sense of Definition~ \ref{DefLBC} with reflex ring $R_{\hat{Z}_i}=:R_i=\kappa_i\dbl\xi_i\dbr$. Let $\BaseFldInSectUnif$ be a field containing all $\kappa_i$. For all $i$ let $\ul{\BL}_i$ be a trivialized local $\BP_i$-shtuka over $\BaseFldInSectUnif$. 
Recall that the Rapoport--Zink space $\breve{\CM}_{\ul{\BL}_i}^{\hat{Z}_i}$ is a formal scheme locally formally of finite type over $R_{\hat{Z}_i}$, see Theorem \ref{ThmRRZSp}. Therefore the product $\breve{\CM}_{(\ul{\BL}_i)_i}^{\hat{\ul Z}}:=\breve{\CM}_{\ul{\BL}_1}^{\hat{Z}_1}\whtimes_\BaseFldInSectUnif\ldots\whtimes_\BaseFldInSectUnif\breve{\CM}_{\ul{\BL}_n}^{\hat{Z}_n}$ is a formal scheme locally formally of finite type over $\Spf\BaseFldInSectUnif\dbl\ul\xi\dbr:=\Spf\BaseFldInSectUnif\dbl\xi_1,\ldots,\xi_n\dbr$. Note that $\breve{\CM}_{(\ul{\BL}_i)_i}^{\hat{\ul Z}}$ is quasi-excellent, see Theorem \ref{ThmRRZSp} and \cite[Th\'eor\`eme 9.2]{ILO}. Recall that the group $J_{\ul{\BL}_i}(Q_i)=\QIsog_\BaseFldInSectUnif(\ul{\BL}_i)$ of quasi-isogenies of $\ul\BL_i$ over $\BaseFldInSectUnif$ acts naturally on $\breve{\CM}_{\ul{\BL}_i}^{\hat{Z}_i}$. Let $J_{(\ul{\BL}_i)_i}:=\prod_i J_{\ul{\BL}_i}(Q_i)$. We say that $\Gamma \subseteq J_{(\ul{\BL}_i)_i}$, which is discrete for the product of the $z_i$-adic topologies, is \emph{separated}, if it is separated in the profinite topology, that is, if for every $1\ne g \in \Gamma$ there is a normal subgroup of finite index that does not contain $g$.\\

The following proposition in particular gives a comparison between the Berkovich's formal nearby cycles and the nearby cycles in the sense of Definition-Remark \ref{Def-RemMFNC} for the Rapoport-Zink sopace $\breve{\CM}_{\ul\BL_0}^{\hat{\ul Z}}$ for local $\BP$-shtukas. Note that for this we require that the local $\BP$-shtuka $\BL_0$ arises from a global $\FG$-shtuka, under the global-local functor.

We let $\FG$ be a parahoric (Bruhat-Tits) group scheme over $C$. Recall that a smooth affine group scheme $\FG$ over C is called a parahoric  group scheme if all geometric fibers of $\FG$ are connected and the generic fiber of $\FG$ is reductive over $\BF_q(C)$, and moreover, for any ramification point $\nu$ of $\FG$ (i.e. those points $\nu$ of $C$, for which the fiber above $\nu$ is not reductive) the group scheme $\BP_\nu := \FG \times_C \Spec A_\nu$ is a parahoric group scheme over $A_\nu$, as defined by Bruhat and Tits \cite{B-TII}[D\'efinition 5.2.6].
Note that every connected reductive group over $Q$ has integral models over $C$ which are parahoric group schemes; see \cite[BT84, $\S\S$ 4.6 and 5.1.9]{B-TII}.

\begin{proposition}\label{PropMiedaCompBerkovich}
Keep the above notation and assume that $(\ul\BL_i)_i:=\hat{\ul\Gamma}(\ul\CG_0)$, for some global $\FG$-shtuka $\ul\CG_0$, where $\hat{\ul\Gamma}(-)$ denote the global-local functor; see \cite[Definition 5.1]{AH_Global}. Here $\FG$ is a parahoric Bruhat-Tits group scheme over $C$, with $\BP_i=\FG_{\nu_i}:=\FG\times_C\wh\CO_{C,\nu_i}$.  For any tuple $\hat{\ul Z}:=(\hat Z_i)_{i=1,\dots,n}$ of boundedness conditions, there is a canonical isomorphism

$$
\pi_{red}^\ast R\Psi_{ \left(\Gamma\backslash\breve{\CM}_{(\ul{\BL}_i)_i}^{\hat{\ul Z}}\right)_\Delta}^\textbf{Ber} \Lambda \to R\Psi_{\left(\breve{\CM}_{(\ul{\BL}_i)_i}^{\hat{\ul Z}}\right)_\Delta}\Lambda,
$$
for a separated discrete subgroup $\Gamma\subseteq J_{(\BL_i)_i}$. Here $\pi$ denotes the projection $\breve{\CM}_{(\ul{\BL}_i)_i}^{\hat{\ul Z}}\to\Gamma\backslash\breve{\CM}_{(\ul{\BL}_i)_i}^{\hat{\ul Z}}$, and the subscript $\Delta$ indicates that these spaces are obtained by pulling back the corresponding spaces under the morphism $\Spf\BaseFldInSectUnif\dbl\xi\dbr\to\Spf\BaseFldInSectUnif\dbl\ul\xi\dbr$, given by $\xi_i \mapsto \xi$.\\

\end{proposition}

\begin{proof}

We first consider the moduli stack $\nabla_n^{H, \ul{\hat{Z}}}\scrH^1(C,\FG)$ for global $\FG$-shtukas bounded by $\ul{\hat{Z}}$, which are equipped with level $H$-structure, for some compact open subgroup $H\subseteq \FG(\BA^\ul\nu)$, with fixed characteristics $\ul\nu:=(\nu_i)_i$; see Remark \ref{RemNablaHData} and also \cite[Definition 7.2]{AH_Global}. We recall from \cite[Chapter 7]{AH_Global} that by the uniformiazation theory of the moduli stacks of global $\FG$-shtukas, we have the following map

\begin{equation}\label{EqUnifMorph}
\Theta\colon  I_{\ul\CG_0}\!(Q) \big{\backslash}\bigl(\breve{\CM}_{(\ul{\BL}_i)_i}^{\hat{\ul Z}}\times \Isom^{\otimes}(\omega^\circ,\check{\CV}_{\ul{\CG}_0})/H\bigr) \;\longto\; \nabla_n^{H,\ulHZ}\scrH^1(C,\FG)^{\ul\nu}\whtimes_{R_{\ulHZ}} \Spf\breve R_{\ulHZ}
\end{equation}
of ind-DM-stacks over $\Spf\breve R_{\ulHZ}$, which is a monomorphism in the sense that the functor $\Theta$ is fully faithful. Here, $I_{\ul\CG_0}\!(Q)$ denote the (abstract) group $\QIsog_\BaseFldInSectUnif(\ul\CG_0)$ of self quasi-isogenies of $\ul\CG_0$, see \cite[Definition~3.4]{AH_Global}, and $\Isom^{\otimes}(\omega^\circ,\check{\CV}_{\ul{\CG}_0})$ denote the set of tensor isomorphisms from neutral fiber functor $\omega^\circ$ to the (dual) Tate functor $\check{\CV}_{\ul{\CG}_0}$, see \cite[Chapter 6]{AH_Global}. Note that $I_{\ul\CG_0}\!(Q)$ acts on the first factor $\breve{\CM}_{(\ul{\BL}_i)_i}^{\hat{\ul Z}}$ via the obvious morphism $I_{\ul\CG_0}\!(Q)\to J_{(\ul\BL_i)_i}$ and on the second factor by its operation on the Tate module $\check{\CV}_{\ul{\CG}_0}$, and furthermore, $H\subseteq \FG(\BA^\ul\nu)\cong \Aut(\omega^\circ)$ operates on the second factor, according to the tannakian formalism. In addition, we recall that the uniformization map induces an isomorphism to certain completion of the stack of global $\FG$-shtukas. Namely, let $\CZ$ be the union of the $\Theta(T_j)$, where $\{T_j\}$ is a set of representatives of $I_{\ul\CG_0}\!(Q)$-orbits of the irreducible components of the scheme $\prod_i X_{Z_i}(\ul\BL_i)\times \Isom^{\otimes}(\omega^\circ,\check{\CV}_{\ul{\CG}_0})/H$ and let $\nabla_n^{H,\ulHZ}\scrH^1(C,\FG)^{\ul\nu}_{/\CZ}$ be the formal completion of $\nabla_n^{H,\ulHZ}\scrH^1(C,\FG)^{\ul\nu}\whtimes_{R_{\ulHZ}} \Spf\breve R_{\ulHZ}$ along $\CZ$, see \cite[Remark~7.12 (a)]{AH_Global}. Here  $X_{Z_i}(\ul\BL_i)$ is the affine Deligne-Lusztig variety corresponding to the local $\BP$-shtuka $\ul\BL_i$ and the boundedness condition $\hat{Z}_i$. Then, the morphism $\Theta$ induces an isomorphism of locally noetherian, adic formal algebraic Deligne-Mumford stacks locally formally of finite type over $\Spf\breve R_{\ulHZ}$
$$
\Theta_{\CZ}\colon  I_{\ul\CG_0}\!(Q) \big{\backslash}\bigl(\breve{\CM}_{(\ul{\BL}_i)_i}^{\hat{\ul Z}}\times \Isom^{\otimes}(\omega^\circ,\check{\CV}_{\ul{\CG}_0})/H\bigr)\;\isoto\; \nabla_n^{H,\ulHZ}\scrH^1(C,\FG)^{\ul\nu}_{/\CZ}\,.
$$

Now notice that one may take the level $H$-structure enough small, such that $\nabla_n^{H, \ul{\hat{Z}}}\scrH^1(C,\FG)$ can be covered by a union of quasi-projective schemes; see \cite[Proposition 3.31]{AH_MR} (also compare \cite[Proposition 2.1]{Var} for the split reductive case). Moreover by \cite[Proposition 7.7]{AH_Global} we have the following decomposition

\begin{equation}\label{EqSourceOfTheta}
I_{\ul\CG_0}\!(Q) \big{\backslash}\bigl(\breve{\CM}_{(\ul{\BL}_i)_i}^{\hat{\ul Z}}\times \Isom^{\otimes}(\omega^\circ,\check{\CV}_{\ul{\CG}_0})/H\bigr) \enspace\cong\enspace \coprod_{\bar\gamma} \Gamma_{\bar\gamma}\big{\backslash}\breve{\CM}_{(\ul{\BL}_i)_i}^{\hat{\ul Z}}\,.
\end{equation}

Here $\bar\gamma:=\gamma H\in\Isom^{\otimes}(\omega^\circ,\check{\CV}_{\ul{\CG}_0})/H$ runs through a set of representatives for the countable double coset $I_{\ul\CG_0}\!(Q) \backslash\Isom^{\otimes}(\omega^\circ,\check{\CV}_{\ul{\CG}_0})/H$, and 
$
\Gamma_{\bar\gamma}\;:=\; I_{\ul\CG_0}\!(Q)\cap \bigl( J_{(\ul{\BL}_i)_i}\times \gamma H \gamma^{-1}\bigr)\;\subset\; \prod_i J_{(\ul{\BL}_i)_i}
$
is a subgroup, which is discrete for the product of the $\nu_i$-adic topologies, and separated in the profinite topology. Concerning this we observe that the quotient spaces $\Gamma_{\bar\gamma}\big{\backslash}\breve{\CM}_{(\ul{\BL}_i)_i}^{\hat{\ul Z}}$ are locally algebrizable. Since the projection map 
\begin{equation}\label{EqProjisAdicandEtale}
\breve{\CM}_{(\ul{\BL}_i)_i}^{\hat{\ul Z}} \to \Gamma_{\bar\gamma}\big{\backslash}\breve{\CM}_{(\ul{\BL}_i)_i}^{\hat{\ul Z}},
\end{equation}
is adic and \'etale; see \cite[Proposition 4.27]{AH_Local}, thus, after restricting to $\Delta$, we may conclude by \cite[Proposition 3.16]{Mieda} and \cite[Th\'eor\`eme 9.2]{ILO}; see also \cite[Proposition 2.12]{Mieda}.\\

\end{proof}

\begin{remark}\label{RemGlobalMethodisInevitable}
Note that along the proof of the above Theorem, we have shown that there is an \'etale and adic morphism $\breve{\CM}_{(\ul{\BL}_i)_i}^{\hat{\ul Z}} \to \Gamma\big{\backslash}\breve{\CM}_{(\ul{\BL}_i)_i}^{\hat{\ul Z}}$, where $\Gamma$ is a subgroup of $J_{(\ul\BL_i)_i}$, and that the quotient $\Gamma\big{\backslash}\breve{\CM}_{(\ul{\BL}_i)_i}^{\hat{\ul Z}}$ is locally algebraizable. Note further that to show this the use of the force of global methods, i.e. the theory of global $\FG$-shukas and their uniformization theory, looks inevitable. For this reason we had to restrict to the case that the local $\BP_{\nu_i}$-shtukas $\ul\BL_i$ are coming from a global $\FG$-shtuka $\ul\CG_0$. Note in addition that this is in fact similar to the analogues situation over number fields, see \cite{F-M04} Corollaire 3.1.5 and  Corollaire 3.2.7 regarding the quasi-algebraizablity, except that one can achieve the algebrizablity only after passing to the quotient by $\Gamma$; compare \cite[Proposition 3.1.3]{F-M04} and also proof of \cite[Theorem 6.23]{RZ}.   

\end{remark}

\begin{remark}\label{RemDirLimCoeff}
Let us set $\FX:= \left(\Gamma\backslash\breve{\CM}_{(\ul{\BL}_i)_i}^{\hat{\ul Z}}\right)_\Delta$. In the situation of the above Proposition, we can see that $\FX_\ol\eta$ admits a covering $\{\CU\}_{\CU\in\BU}$ by quasi-compact open formal subschemes which are also taut. The reason is that as in the proof of the above Proposition $\nabla_n^{H,\ul{\hat{Z}}}\scrH^1(C,\FG)$ admits a covering by quasi-compact opens which in turn induces a covering for $\nabla_n^{H,\ul{\hat{Z}}}\scrH^1(C,\FG)_{/\CZ}$ and thus for $\FX_{\ol\eta}$ by quasi-compact opens $\CU$, which are special in the sense of \cite{Berk}. This ensures that they are taut, see \cite[Lemma 4.18]{Mieda}. Therefore we observe that  $H_c^i(\FX_\ol\eta, \BZ_\ell)\cong \dirlim[\CU]H_c^i(\CU, \BZ_\ell)
$, see Proposition 2.1.~iv) of \cite{Hub98}. Note further that one may equip the analytic space corresponding to the generic fiber of the formal scheme $\breve{\CM}_{\ul\BL}^{\hat{Z}}$ with level $H$-structure, for compact open subgroup $H\subseteq \FG(\BF_q\dpl z\dpr)$; see \cite{HV21}. This can be achieved by trivializing the universal local $\BP$-shtuka.

\end{remark}

\begin{proposition}
 Keep the assumption and notation in Proposition \ref{PropMiedaCompBerkovich}. Then there is a natural isomorphism 

$$
H_c^i(\FX_\ol\eta, \Lambda)\to H_c^i(\FX_{red}, R\Psi_{\FX,c}\Lambda),
$$
where $\FX:= \left(\Gamma\backslash\breve{\CM}_{(\ul{\BL}_i)_i}^{\hat{\ul Z}}\right)_\Delta$ and $\FX_{red}=\Gamma\backslash \prod_i X_Z(\ul\BL_i)$.

\end{proposition}

\begin{proof}
Regarding the proof of Proposition \ref{PropMiedaCompBerkovich} we observe that $\FX$ locally algebraizable; see also remark \ref{RemGlobalMethodisInevitable}. The statement now follows from \cite[Theorem 1.1.iv)]{Mieda},  Theorem \ref{ThmRRZSp} and the argument given in Remark \ref{RemDirLimCoeff} together with \cite[Example 4.25. ii))]{Mieda}
\end{proof}

\begin{corollary}\label{Cor_SemisimpleTrace}
Let $\nu$ be a place on $C$ and set $\BP:=\BP_\nu$. Let $\ul\BL$ be a local $\BP$-shtuka, which comes from a global $\FG$-shtuka $\ul\CG$ under the functor $\hat{\Gamma}_{\nu}(-)$. Set $\breve{\CM}:=\breve{\CM}_{\ul\BL}^{\hat{Z}}$ and $\kappa=\kappa_{\hat{Z}}$. We have the following statements

\begin{enumerate}
    \item 
There is a canonical isomorphism $R\Psi_{\Gamma \backslash \breve\CM}\Lambda \cong R\Psi_{\Gamma \backslash \breve\CM}^{\textbf{Ber}}\Lambda$ for a separated discrete subgroup $\Gamma\subseteq J(\ul\BL)$.
\item
Let $\kappa_r/\kappa$ be a finite extension of degree $r$. Let $x$ be a point in $\breve{\CM}(\kappa_r)$ and let $y$ be the image $\pi^{loc}(y')$ of a point $y'$ in $\wt\CM_{\ul\BL}^{\hat{Z}}$ above $x$ under $\pi$ (see the local model roof (\ref{EqnablaHRoof})). Then 
 
$$
tr^{ss}\left(Frob_r; \left(R\Psi_{\breve{\CM},c}\ol\BQ_\ell\right)_x\right)=tr^{ss}\left(Frob_r; \left(R\psi_{\hat{Z}}\ol\BQ_\ell\right)_y\right).
$$
Here $Frob_r$ denotes the geometric Frobenius in $\Gal(\ol\kappa_r/\kappa_r)$.  
\end{enumerate}

\end{corollary}

\begin{proof}

a) Recall from proof of Proposition \ref{PropMiedaCompBerkovich} that for some separated discrete subgroup $\Gamma\subseteq J(\ul\BL)$, the projection $\breve\CM\to\Gamma\backslash\breve{\CM}$ is adic and \'etale. The second one is locally noetherian quasi-excellent formal algebraic scheme, see \cite[Proposition 2.12]{Mieda}, with separated underlying reduced subscheme, see Theorem \ref{ThmRRZSp}, which is in addition locally algebraizable, see the proof of Proposition \ref{PropMiedaCompBerkovich}; see also Remark \ref{RemGlobalMethodisInevitable}. Now the isomorphism follows from \cite[Theorem 1.1.iii)]{Mieda}.\\

b) This statement follows from Remark \ref{RemGlobalMethodisInevitable} and the following roof

\begin{equation*}
\xygraph{
!{<0cm,0cm>;<1cm,0cm>:<0cm,1cm>::}
!{(0,0) }*+{\CM'}="a"
!{(-0.75,-1) }*+{\breve{\CM}}="b"
!{(1,-2) }*+{\hat{Z},}="c"
!{(-1.5,-2) }*+{\Gamma\slash\breve{\CM}}="d"
"a":^{\text{\'et}}"b" "a":^{\pi^{loc}\circ s'}"c"
"b":^{\text{\'et}}"d"
}  
\end{equation*}
of \'etale morphisms; see the diagram \ref{EqM'}, constructed in the course of the proof of theorem \ref{ThmRapoportZinkLocalModel}, and the explanation given about the morphism (\ref{EqProjisAdicandEtale}). 
Note that since $\BP$ is parahoric $\wh{\CF\ell}_\BP$ is ind-proper, and thus $\hat{Z}$ is algebraizable, in particular $R\Psi_{{\hat{Z}},c}\ol\BQ_\ell$ is isomorphic to the usual nearby cycles sheaf $R\psi_{\hat{Z}}\ol\BQ_\ell$ for schemes, see proof of Proposition \ref{PropnearbycyclesandhatZ}. 

\end{proof}

\begin{minipage}[t]{1\linewidth}
\noindent
\small{}Esmail Arasteh Rad, School of Mathematics, Institute for Research in Fundamental Sciences (IPM), P.O.
Box: 19395-5746, Tehran, Iran.
\emph{email}: \href{mailto:earasteh@ipm.ir}{earasteh@ipm.ir}

\end{minipage}

{}


{\small
\begin{thebibliography}{GHKR2}
\addcontentsline{toc}{section}{References}






\bibitem[Ara12]{EsmailDissertation} E.~Arasteh Rad, \emph{Uniformizing The Moduli Stacks of Global $\FG$-Shtukas}, dissertation, Universit\"at M\"unster, 2012.


\bibitem[AH14]{AH_Local} E.~Arasteh Rad, U.~Hartl: \emph{Local $\BP$-shtukas and their relation to global $\FG$-shtukas}, Muenster J.~Math {\bfseries 7} (2014), 623--670; open access at \href{http://dx.doi.org/10.17879/58269757072}{http:/\hspace{-1mm}/miami.uni-muenster.de}.


\bibitem[AraHar16]{AH_LR} E.~Arasteh Rad, U.~Hartl: \emph{On Langlands-Rapoport Conjecture over Function Fields}, preprint 2016 \url{http://arxiv.org/abs/1605.01575}.

\bibitem[AraHar19]{AH_Global} E.~Arasteh Rad, U.~Hartl: \emph{Uniformizing the moduli stacks of global $\FG$-Shtukas}, Int.\ Math.\ Res.\ Notices Volume 2021, Issue 21, Pages 16121--16192 (2019); also available as \url{http://arxiv.org/abs/1302.6351}.


\bibitem[AraHar20]{A_CMot} E.~Arasteh Rad and U. Hartl \emph{Category of $C$-Motives over Finite Fields}, Journal of Number Theory 2020, DOI: https://doi.org/10.1016/j.jnt.2020.06.015, preprint available as \url{http://arxiv.org/abs/1810.11941}.



\bibitem[AraHab19a]{AH_LM} E.~Arasteh Rad and S.~Habibi:  \emph{Local Models For the Moduli Stacks of Global G-Shtukas}, Mathematical Research Letters Vol. 26, No. 2 (2019), pp. 323-364. preprint on \url{http://arxiv.org/abs/1605.01588v3}.

\bibitem[AraHab19b]{AH_MR} E.~Arasteh Rad and S.~Habibi: \emph{Motivic Remarks On The Moduli Stacks Of global G-Shtukas And Their Local Models}, preprint 2019, \url{http://arxiv.org/abs/1912.09968}.















\bibitem[BD]{B-D} A.~Beilinson, V.~Drinfeld: \emph{Quantization of Hitchin's integrable system and Hecke eigensheaves}, preprint on \url{http://www.math.uchicago.edu/~mitya/langlands/hitchin/BD-hitchin.pdf}.


\bibitem[BBD]{BBD}
A.~Beilinson, J.~Bernstein, P.~Deligne: \emph{Faisceaux pervers}, 5--171, Ast\'erisque, 100, Soc.~Math. France, Paris, 1982.



\bibitem[BerI]{BerkI} V.~Berkovich, \emph{Vanishing cycles for formal schemes}, Invent. Math.,
115(3) (1994), pp. 539--571

\bibitem[BerII]{Berk} V.~Berkovich, \emph{Vanishing cycles for formal schemes II}, Invent. Math., 125(2), 367--390, 1996.






\Verkuerzung
{
}
{}

\bibitem[Bre]{Paul} P.~ Breutmann : \emph{Functoriality of Moduli Spaces of Global G-Shtukas} preprint 2019 \url{http://arxiv.org/abs/1902.10602}.

\bibitem[BT72]{B-T} F.~Bruhat, J.~Tits: \emph{Groupes r\'eductifs sur un corps local}, Inst.\ Hautes \'Etudes Sci.\ Publ.\ Math.\ {\bfseries 41} (1972), 5--251.

\bibitem[BT84]{B-TII} F.~Bruhat, J.~Tits: \emph{Groupes r\'eductifs sur un corps local II.~Sch\'emas en groupes, Existence d'une donn\'e radicielle valu\'ee}, Inst.\ Hautes \'Etudes Sci.\ Publ.\ Math.\ {\bfseries 60} (1984), 197--376.

\Verkuerzung
{
}
{}





\bibitem[Del70]{Deligne1} P.\ Deligne, \emph{Travaux de Griffiths}, S\'eminaire Bourbaki (1969/1970), Exp. no 376.

\bibitem[Del71]{Deligne2} P.\ Deligne, \emph{Travaux de Shimura}, S\'eminaire Bourbaki (1970/1971), Exp. no. 389.







\bibitem[Fal03]{Faltings03} G.~Faltings: \emph{Algebraic loop groups and moduli spaces  of bundles}, J.\ Eur.\ Math.\ Soc.\ {\bfseries 5} (2003), no.~1, 41--68. 


\bibitem[F-M04]{F-M04} L. Fargues and E. Mantovan, \emph{Vari\'et\'s de Shimura, espaces de Rapoport-Zink et correspondances de Langlands locales}, Ast\'erisque, no. 291 (2004), 343 p. \url{http://www.numdam.org/item/AST_2004__291__R1_0.pdf}


\bibitem[EGA]{EGA} A.~Grothendieck: \emph{{\'E}lements de G{\'e}o\-m{\'e}trie Alg{\'e}\-brique}, Publ.\ Math.\ IHES {\bfseries 4}, {\bfseries 8}, {\bfseries 11}, {\bfseries 17}, {\bfseries 20}, {\bfseries 24}, {\bfseries 28}, {\bfseries 32}, Bures-Sur-Yvette, 1960--1967; see also Grundlehren {\bfseries 166}, Springer-Verlag, Berlin etc.\ 1971; also available at \href{http://www.numdam.org/numdam-bin/recherche?au=Grothendieck}{http:/\!/www.numdam.org/numdam-bin/recherche?au=Grothendieck}.







\bibitem[SGAIV]{SGA4} A.~Grothendieck: \emph{Th\'eorie des topos et cohomologie \'etale des sch\'emas}, Lecture Notes in
Mathematics, Volume 269, 270, 305 (Springer-Verlag, Berlin, 1972–1973).




\bibitem[HR18a]{HR1} T.~Haines, T.~Richarz  \emph{The test function conjecture for parahoric local models}, J. Amer. Math. Soc. 34 (2021), 135-218, also available as \url{http://arxiv.org/abs/1801.07094}.




\bibitem[HR03]{H-R} T.~Haines and M.~Rapoport: \emph{On parahoric subgroups}, appendix to \cite{PR2}; also available as \url{http://arxiv.org/abs/0804.3788}.

\bibitem[Har05]{Har1} U.~Hartl: \emph{Uniformizing the Stacks of Abelian Sheaves}, in Number Fields and Function fields - Two Parallel Worlds, Papers from the 4th Conference held on Texel Island, April 2004, Progress in Math.\ {\bfseries 239}, Birkh\"auser-Verlag, Basel 2005, pp.~167--222; also available as \url{http://arxiv.org/abs/math/0409341}.











\bibitem[Hub98]{Hub98} R. Huber, \emph{A comparison theorem for $\ell$-adic cohomology}, Compositio Math. 112(2)
(1998), 217--235.

\bibitem[HV11]{H-V} U.~Hartl, E.~Viehmann: \emph{The Newton stratification on deformations of local $G$-shtukas}, J.\ reine angew.\ Math.\ (Crelle) {\bfseries 656} (2011), 87--129; also available as \url{http://arxiv.org/abs/0810.0821}.


\bibitem[HV21]{HV21} U. Hartl, E. Viehmann, \emph{The generic fiber of moduli spaces of bounded local G-shtukas}, Journal of the Institute of Mathematics of Jussieu (2021), 1-80. doi:10.1017/S1474748021000293, available at \href{https://arxiv.org/pdf/1712.07936.pdf}{arXiv:1712.07936}   


\Verkuerzung
{
}
{}



\bibitem[ILO]{ILO} L. Illusie, Y. Laszlo and F. Orgogozo, \emph{Travaux de Gabber sur l'uniformisation
locale et la cohomologie \'etale des sch\'emas quasi-excellents}, S\'eminaire \`a l' \'ecole
polytechnique 2006–2008, preprint, (\url{http://www.math.polytechnique.fr/~orgogozo/travaux de Gabber}), 2012.





\bibitem[Kim]{Kim} W.~Kim, \emph{Rapoport-Zink spaces of Hodge type},  Forum of Mathematics, Sigma, 6, E8. doi:10.1017/fms.2018.6, also available as: \url{http://arxiv.org/abs/1308.5537}.

\bibitem[Kis]{Kis} M.~Kisin, \emph{Integral models for Shimura varieties of abelian type}, J. Amer. Math. Soc. 23 (2010), 967-1012.



\Verkuerzung
{
}
{}

\bibitem[LLaf]{Laff} L.\ Lafforgue, \emph{Chtoucas de Drinfeld et correspondance de Langlands}, Invent.~ Math.~ 147 (2002), 1--241; also available at \url{http://www.ihes.fr/~lafforgue/math/fulltext.pdf}.


\bibitem[VLaf]{VLaff} V.~Lafforgue. \emph{Chtoucas pour les groupes r\'eductifs et param\'etrisation de Langlands globale.} J. Amer. Math. Soc. 31 (2018), no. 3, 719–891; Preprint (2012) \url{http://arxiv.org/abs/1209.5352}.












\bibitem[Lev]{Lev} B. Brandon Levin,  \emph{Local models for Weil-restricted groups}, Compositio Mathematica, 152(12), 2563-2601. doi:10.1112/S0010437X1600765X, also available as: \url{https://arxiv.org/pdf/1412.7135.pdf}.

\bibitem[Mat]{Matsumura} H.~Matsumura, \emph{Commutative Ring Theory}, Cambridge studies in advanced mathematics 8, Cambridge Univ. Press, 1986.


\bibitem[Mie]{Mieda} Y. Mieda, \emph{Variants of formal nearby cycles}, J. Inst. Math. Jussieu 13:4 (2014), 701--751.



\bibitem[Mil]{MilneShimura} J.~Milne \emph{Introduction to Shimura varieties}. Clay Math. Proceedings 4 (2005) p. 265--378.



\bibitem[NP]{Ngo-Polo} B.C.\ Ng\^o, P.\ Polo: \emph{R\'esolutions de Demazure affines et formule de Casselman-Shalika g\'eom\'etrique}, J.\ Algebraic Geom.\ {\bfseries 10} (2001), no. 3, 515--547; also available as \url{http://www.math.uchicago.edu/~ngo/ngo-polo.pdf}. 

\bibitem[PR03]{PR1} G.\ Pappas, M.\ Rapoport: \emph{Local models in the ramified case I, The EL-case}, J.\ Algebraic Geom.~{\bfseries 12} (2003), no.\ 1, 107--145; also available as \url{http://arxiv.org/abs/math/0006222}. 

\bibitem[PR08]{PR2} G.~Pappas, M.~Rapoport: \emph{Twisted loop groups and their affine flag varieties}, Advances in Math.\ {\bfseries 219} (2008), 118--198; also available as \url{http://arxiv.org/abs/math/0607130}.

\bibitem[PRS]{PRS} G.~Pappas, M.~Rapoport, B.~Smithling: \emph{Local models of Shimura varieties, I. Geometry and combinatorics} in Handbook of moduli. Vol. III, Adv. Lect. Math. (ALM), vol. 26, Int. Press,
Somerville, MA, 2013, p. 135–217.


\bibitem[PZ]{PZ}  G.~ Pappas, X.~ Zhu, \emph{Local models of Shimura varieties and a conjecture of Kottwitz}, Invent.
Math.\ {\bfseries 194} (2013), no.1, 147--254; also available at \url{http://arxiv.org/abs/1110.5588}.




\bibitem[Rap94]{Rap94} M.~ Rapoport, \emph{Non-archimedean period domains},  Proceedings of the International Congress of Mathematicians,
Vol. 1, 2, Z\"urich, 1994, 423-434

\bibitem[RV]{R-V}  M. Rapoport and E. Viehmann, \emph{Towards a theory of local Shimura varieties}, M\"unster J. Math. 7 (2014), no. 1, 273--326. MR 3271247.

\bibitem[RZ96]{RZ} M.\ Rapoport, T.\ Zink: {\it Period Spaces for $p$-divisible Groups\/}, Ann.\ Math.\ Stud.\ {\bf 141}, Princeton University Press, Princeton 1996.





\bibitem[Ri13a]{Richarz} T.~Richarz: \emph{Schubert varieties in twisted affine flag varieties and local models}, J.~Algebra {\bfseries 375} (2013), 121--147; also available as \url{http://arxiv.org/abs/1011.5416}.


\bibitem[Ri16b]{Richarz16} T.~Richarz: \emph{Affine Grassmannians and Geometric Satake Equivalences}, , IMRN 2016, 3717--3767, preprint available at \url{http://arxiv.org/abs/1311.1008}.

\bibitem[RV]{RV} M.~Rapoport and E.~ Viehmann, \emph{Towards a theory of local Shimura varieties}, M\"unster J. Math.
7 (2014), no. 1, 273--326. 

\bibitem[She]{Shen} X.~Shen, \emph{On some generalized Rapoport-Zink Spaces}, Canadian Journal of Mathematics, 72(5), 1111--1187. doi:10.4153/S0008414X19000269, also available as \url{https://arxiv.org/pdf/1611.08977.pdf}.

\bibitem[SW]{SchWei} P.~ Scholze, J.~ Weinstein \emph{Berkeley lectures on p-adic geometry}, Annals of
Math., vol. 207, Princeton University Press, 2020. Also available at  \url{http://www.math.uni-bonn.de/people/scholze/Berkeley.pdf}



\Verkuerzung
{
}
{}


\bibitem[Var04]{Var} Y.\ Varshavsky: \emph{Moduli spaces of principal $F$-bundles}, Selecta Math.\ (N.S.) {\bfseries 10} (2004),  no.\ 1, 131--166; also available as \url{http://arxiv.org/abs/math/0205130}.









\bibitem[Zhu14]{Zhu} X. Zhu, \emph{On the coherence conjecture of Pappas and Rapoport}, Annals of
Math., Volume 180 (2014), Issue 1, Pages 1-85. 

\end{thebibliography}
}
\end{document}